\documentclass[a4paper,11pt]{amsart}
\usepackage[utf8]{inputenc}
\usepackage{amsmath,amsthm,amssymb,mathtools}
\usepackage[shortlabels]{enumitem}
\usepackage{mathrsfs}
\usepackage{xcolor}
\usepackage[colorlinks=true,linkcolor=blue!50!black,citecolor=blue!50!black,urlcolor=blue!50!black]{hyperref}
\usepackage[capitalize,nameinlink]{cleveref}
\newtheorem{theorem}{Theorem}[section]
\newtheorem{proposition}[theorem]{Proposition}
\newtheorem{lemma}[theorem]{Lemma}

\theoremstyle{definition}
\newtheorem{definition}[theorem]{Definition}
\newtheorem{example}[theorem]{Example}
\newtheorem{setting}[theorem]{Setting}
\theoremstyle{remark}
\newtheorem{remark}[theorem]{Remark}
\crefname{setting}{Setting}{Settings}

\newcommand{\C}{\mathbb{C}}
\newcommand{\Z}{\mathbb{Z}}
\newcommand{\N}{\mathbb{N}}
\newcommand{\PP}{\mathbb{P}}
\newcommand{\cO}{\mathcal{O}}
\newcommand{\cI}{\mathcal{I}}
\newcommand{\cF}{\mathcal{F}}
\newcommand{\cS}{\mathscr{S}}
\newcommand{\pr}{\mathrm{pr}}
\newcommand{\Op}{\mathrm{Op}}
\newcommand{\merto}{\dashrightarrow}
\newcommand{\Khat}{\widehat{K}}
\newcommand{\GO}{\Gamma_{\cO}}
\newcommand{\GP}{\Gamma^{P}_{\cO}}
\DeclareMathOperator{\Pic}{Pic}
\DeclareMathOperator{\Sym}{Sym}
\DeclareMathOperator{\dist}{dist}
\DeclareMathOperator{\id}{id}
\DeclareMathOperator{\supp}{supp}

\begin{document}

\title[Meromorphic approximation of holomorphic maps]{Meromorphic approximation of\\ holomorphic maps into\\ algebraically elliptic projective varieties}

\author[Y. Kusakabe]{Yuta Kusakabe}
\address{Faculty of Mathematics, Kyushu University, Fukuoka 819-0395, Japan}
\email{kusakabe@math.kyushu-u.ac.jp}
\author[T. Miyazaki]{Taichi Miyazaki}
\address{Joint Graduate School of Mathematics for Innovation, Kyushu University, Fukuoka 819-0395, Japan}
\email{miyazaki.taichi.478@s.kyushu-u.ac.jp}

\begin{abstract}
We show that holomorphic maps from a neighborhood of a compact set in a Stein space to an algebraically elliptic smooth projective variety can be approximated by meromorphic maps, provided that the map extends continuously to the whole space and the compact set is convex with respect to a submonoid of the Picard group. The proof uses algebraic sprays and their polynomial lifts to affine cones, which we call conic lifts. Rational sprays, which generalize algebraic sprays, extend the result to maps avoiding the exceptional set of a rational spray. We also show that, for smooth projective varieties, the existence of an algebraic spray which is dominating at some point is a birational invariant.
\end{abstract}

\maketitle

\section{Introduction}\label{sec:intro}

\subsection{Background and notation}
The parametric Oka principle with approximation and interpolation (POPAI) deforms continuous families of continuous maps from Stein spaces to Oka manifolds into continuous families of holomorphic maps, with approximation on holomorphically convex compact sets and interpolation on closed analytic subvarieties and closed parameter sets \cite{Gr89}, \cite[Theorem 5.4.4]{F17}, \cite[Theorem 1.2]{F23}, \cite[Corollary 5.7]{Ku21}. We seek an analog with approximation by meromorphic maps on compact sets that need not be holomorphically convex.

Approximation of holomorphic functions by meromorphic ones is classical, going back to Runge's theorem in one variable (see \cite[Theorem 10]{Ro67} on compact Riemann surfaces) and to the Oka--Weil theorem on rationally convex sets in $\C^n$. On Stein manifolds, Hirschowitz \cite[Th\'eor\`eme 2]{Hi71} treated compact sets which are convex with respect to hypersurfaces, and Abe \cite[Theorem 11]{Ab05} proved a meromorphic version of the Oka--Weil theorem on reduced Stein spaces. Abe \cite{Ab09} unified both results by introducing convexity with respect to a subsemigroup $G$ of the Picard group of a Stein space $X$ (\cref{def:abe-hull}). On a $G$-convex compact set, every holomorphic function can be approximated by quotients of global sections of a line bundle in $G$ \cite[Theorem 5.1]{Ab09}. Meromorphic convexity has also been studied recently \cite{BS26,BS25}. For maps into projective space, Zeron \cite[Theorem 2]{Ze06} showed that a holomorphic map from a neighborhood of a rationally convex compact set $K\subset\C^n$ with $\check H^2(K;\Z)=0$ to $\PP^m$ can be approximated uniformly on $K$ by maps $[P]$, where $P$ is a polynomial map without zeros on $K$ (see also \cite{GZ06}). Our results extend these approximation theorems, with parameters and interpolation, to maps into algebraically elliptic smooth projective varieties (\cref{def:alg-spray}) and, away from the exceptional set of a rational spray, into rationally elliptic ones (\cref{def:intro-spray}). Holomorphic maps from Runge domains in affine varieties to quasi-projective manifolds can be approximated by Nash algebraic maps \cite[Theorem 1.1]{DLS94}, see also \cite[Theorem 1.1]{Le95}. For compact targets, approximation by regular maps from affine varieties fails in general \cite[Theorem 2(b)]{LT19}, so meromorphic maps are natural approximants.

In this paper, all complex spaces are assumed to be reduced and second countable, and all algebraic varieties are assumed to be irreducible. For a subset $K\subset X$ we write $\Op K$ for an unspecified open neighborhood of $K$. For a compact set $K$ we write $\cO(K,Y)$ for the set of maps to $Y$ which are holomorphic on an open neighborhood of $K$. For a complex space $X$, $\cO(X,Y)$ denotes the set of holomorphic maps $X\to Y$. We write $\dist$ for the Fubini--Study distance on $\PP^N$ and $\N=\{1,2,\dots\}$.

\subsection{Main theorems}

\begin{theorem}\label[theorem]{thm:main1}
Let $X$ be a Stein space, $G\subset\Pic(X)$ a submonoid, and $Y\subset\PP^N$ an algebraically elliptic smooth projective variety \textup{(}\cref{def:alg-spray}\textup{)}. Let $K\subset X$ be a compact $G$-convex set \textup{(}\cref{def:abe-hull}\textup{)}, $X_0\subset X$ a closed analytic subvariety, and $(P,P_0)$ a finite CW pair. Let $f_0\colon P\times X\to Y$ be a continuous map satisfying the following conditions:
\begin{itemize}
\item $f_0(p,\cdot)\in\cO(X,Y)$ for all $p\in P_0$, and $f_0(p,\cdot)|_{X_0}\in\cO(X_0,Y)$ for all $p\in P$,
\item $f_0|_{P\times\Op K}$ is a continuous family of holomorphic maps.
\end{itemize}
Then there exist $L\in G$ and $\sigma\in\GO(X,L)$ with $\{\sigma=0\}\cap K=\emptyset$, vanishing identically on no irreducible component of $X$ or of $X_0$, with the following property. For each $\epsilon>0$ there are a continuous family of meromorphic maps $F=[\widetilde F]\colon P\times X\merto Y$ \textup{(}\cref{def:families}\textup{)} such that $\widetilde F$ has no zeros on $P\times(X\setminus\{\sigma=0\})$, and a homotopy of continuous maps $f_t\colon P\times(X\setminus\{\sigma=0\})\to Y$ \textup{(}$t\in[0,1]$\textup{)} from $f_0|_{P\times(X\setminus\{\sigma=0\})}$ to $f_1:=F|_{P\times(X\setminus\{\sigma=0\})}$, such that for all $t\in[0,1]$:
\begin{enumerate}[(i)]
\item $f_t(p,\cdot)|_K\in\cO(K,Y)$ for all $p\in P$, and $\sup_{P\times K}\dist(f_t,f_0)<\epsilon$,
\item $f_t=f_0$ on $P_0\times(X\setminus\{\sigma=0\})$ and on $P\times(X_0\setminus\{\sigma=0\})$.
\end{enumerate}
In particular $F(p,\cdot)=f_0(p,\cdot)$ as meromorphic maps for $p\in P_0$, and $F(p,\cdot)=f_0(p,\cdot)$ on $X_0\setminus\{\sigma=0\}$, which is dense in $X_0$, for all $p\in P$. If moreover $X$ is a Stein manifold, $X_0$ a closed complex submanifold, $f_0|_{P\times\Op X_0}$ is a continuous family of holomorphic maps, and $\ell\ge1$, then $F$ can be chosen such that $F(p,\cdot)$ and $f_0(p,\cdot)$ agree to order $\ell$ along $X_0\setminus\{\sigma=0\}$.
\end{theorem}

The proof combines the method of proof of \cref{thm:main2} with the Oka principle (\cref{sec:proof-main1}).

Let $X$, $G$, $Y$, $K$, $X_0$ and $P_0\subset P$ be as above, $\ell\ge1$ an integer, and $f_1\colon P\times\Op K\to Y$ a continuous family of holomorphic maps. Consider the following conditions.
\begin{enumerate}
\item[(A)] For every $\epsilon>0$ there exists a continuous family of meromorphic maps $F=[\widetilde F]\colon P\times X\merto Y$ (\cref{def:families}) such that
\begin{itemize}
\item $\widetilde F$ has no zeros on $P\times\Op K$, so $F(p,\cdot)|_K\in\cO(K,Y)$ for all $p\in P$,
\item $\sup_{P\times K}\dist(F,f_1)<\epsilon$,
\item $F(p,\cdot)$ and $f_1(p,\cdot)$ agree to order $\ell$ along $X_0\cap\Op K$ for all $p$, and $F(p,\cdot)=f_1(p,\cdot)$ on $\Op K$ for $p\in P_0$.
\end{itemize}
\item[(H)] There exist a continuous family of meromorphic maps $f_0=[\tilde f_0]\colon P\times X\merto Y$ and a continuous family $f_t\colon P\times\Op K\to Y$ \textup{(}$t\in[0,1]$\textup{)} of holomorphic maps from $f_0|_{P\times\Op K}$ to $f_1$ such that for all $(t,p)\in[0,1]\times P$:
\begin{itemize}
\item $\tilde f_0$ has no zeros on $P\times\Op K$,
\item $f_t(p,\cdot)$ and $f_0(p,\cdot)$ agree to order $\ell$ along $X_0\cap\Op K$,
\item $f_t(p,\cdot)=f_1(p,\cdot)$ if $p\in P_0$.
\end{itemize}
\end{enumerate}

\begin{theorem}\label[theorem]{thm:main2}
Let $X$, $G$, $Y$, $K$, $X_0$ and $P_0\subset P$ satisfy the hypotheses of \cref{thm:main1}, let $\ell\ge1$ be an integer, and let $f_1\colon P\times\Op K\to Y$ be a continuous family of holomorphic maps. Then:
\begin{enumerate}[(1)]
\item Conditions \textup{(A)} and \textup{(H)} are equivalent.
\item Assume \textup{(H)}. Then there exist $L\in G$ and $\sigma\in\GO(X,L)$ with $\{\sigma=0\}\cap K=\emptyset$ and $\widehat K_{X\setminus\{\sigma=0\}}\subset\Op K$, vanishing identically on no irreducible component of $X$ or of $X_0$, such that for every $\epsilon>0$ there exist an integer $k\ge0$, a continuous family of meromorphic maps $F=[\widetilde F]\colon P\times X\merto Y$ and a continuous family of holomorphic maps $\hat f_t\colon P\times\Op K\to Y$ \textup{(}$t\in[0,1]$\textup{)} such that for all $t\in[0,1]$ and $p\in P$:
\begin{enumerate}[(i)]
\item $\widetilde F$ has no zeros on $P\times\Op K$, $\hat f_0=F|_{P\times\Op K}$ and $\hat f_1=f_1$,
\item $\hat f_t(p,\cdot)=f_1(p,\cdot)$ if $p\in P_0$,
\item $\hat f_t(p,\cdot)$ and $f_1(p,\cdot)$ agree to order $\ell$ along $X_0\cap\Op K$,
\item $\sup_{P\times K}\dist(\hat f_t,f_1)<\epsilon$,
\item $L(F)=L(f_0)\otimes\pr_X^*L^k$ \textup{(}\cref{def:families}\textup{)},
\item if $p\in P_0$, then $\widetilde F(p,\cdot)=\sigma^k\tilde f_0(p,\cdot)$. In particular $F(p,\cdot)=f_0(p,\cdot)$,
\item the components of $\widetilde F(p,\cdot)-\sigma^k\tilde f_0(p,\cdot)$ belong to $\Gamma\bigl(X,\cI_{X_0}^{\ell}\cdot\cO(L(F)_p)\bigr)$, where $\cI_{X_0}$ is the ideal sheaf of $X_0$.
\end{enumerate}
\end{enumerate}
\end{theorem}

For every $p\in P$, $\widetilde F(p,\cdot)$ has no zeros on $X\setminus\{\sigma=0\}$ wherever $\tilde f_0(p,\cdot)$ has none, and $F(p,\cdot)$ is holomorphic on $X\setminus\{\sigma=0\}$ wherever $f_0(p,\cdot)$ is (\cref{sec:proof-main2}). If $K$ is $\cO(X)$-convex, then $\sigma=1$. Condition \textup{(vii)} is a congruence of tuples and does not assert that $F(p,\cdot)$ is holomorphic at points of $X_0$ where $\tilde f_0(p,\cdot)$ or $\sigma$ vanishes.

\begin{remark}\label[remark]{rem:continuous-H}
For $\ell=1$, condition (H) is equivalent to the condition obtained by requiring $f_t$ only to be a homotopy of continuous maps. Indeed, since $Y$ is Oka, POPAI \cite[Theorem 1.2]{F23} on a Stein neighborhood of $K$, relative to the endpoints, $P_0$ and $X_0$, deforms such a homotopy into a homotopy of holomorphic maps, as in \cref{sec:proof-main1}.
\end{remark}

\begin{remark}\label[remark]{rem:parameters}
Except for the interpolation of jets in \cref{thm:main1}, \cref{thm:main1,thm:main2,thm:general} hold, with the same proofs, for compact Hausdorff spaces $P_0\subset P$ such that every point of $P$ has a contractible compact neighborhood. Indeed, POPAI holds for such parameter spaces \cite[Theorem 1.2]{F23}, see also \cite[Corollary 5.7]{Ku21}, and \cref{prop:runge-interp} is stated in this generality. The interpolation of jets in \cref{thm:main1} uses \cite[Theorem 1.4]{FP01}, which requires $P_0$ to be a strong deformation retract of a neighborhood of $P_0$ in $P$.
\end{remark}

\Cref{thm:main1} does not hold for analytically elliptic manifolds in general (\cref{rem:torus}).

When $X$ is an open Riemann surface, every compact set is $G$-convex for $G=\{[\cO_X]\}$, since for $x\notin K$ there is a holomorphic function on $X$ whose only zero is a simple zero at $x$, and meromorphic maps from $X$ to $Y$ are holomorphic. For $P$ a point, \cref{thm:main1} then concerns approximation by holomorphic maps on arbitrary compact sets. Since $Y$ is unirational (\cref{thm:ME}), hence rationally connected, this case is also covered by the concurrent work \cite[Theorem 1.5]{DGX26} (with a point as base, see also \cite[Corollary 1.8]{DGX26}). Oka-1 manifolds were introduced in \cite{AF25}, see also \cite[Theorem 4.1, Corollary 1.8]{FL25} for the algebraic Oka-1 property and \cite{BW25} for a related algebraic approximation theorem.

\subsection{Rational sprays and rational ellipticity}
We adapt the spray-lifting method of \cite[Theorem 3.1]{F06} (see also \cite[\S 6.15]{F17}), where holomorphic maps homotopic to algebraic maps are approximated by algebraic maps. Bochnak and Kucharz \cite{BK23} use dominating cascades to approximate holomorphic maps defined near compact subsets of affine varieties by rational maps which are regular near the compact set. A dominating cascade \cite[Definition 2.8]{BK23} is a rational spray on a trivial bundle which is regular near the zero section and has empty exceptional set, and \cite[Theorem 2.9]{BK23} is a prototype of \cref{thm:general} with $Z=\emptyset$. Rational sprays also allow the twisted bundles $\cO_Y(-m)^{\oplus r}$, which arise naturally for projective $Y$ (\cref{prop:normalized}), and a nonempty exceptional set. In \cite{FL26}, dominating holomorphic sprays on projective Oka manifolds are constructed on direct sums of copies of the dual of a sufficiently ample line bundle. Here rational sprays and conic lifts allow substitution of meromorphic data in both the base and fiber variables, yielding global meromorphic maps from Stein spaces.

\begin{definition}\label[definition]{def:intro-spray}
Let $Y$ be a smooth variety and $p\colon E\to Y$ an algebraic vector bundle with zero section $0_E\colon Y\to E$. For a rational map $f$ we denote by $I(f)$ its \emph{indeterminacy set}, the complement of the largest Zariski open set on which $f$ is regular. A \emph{rational spray} on $Y$ is a rational map $s\colon E\merto Y$ such that $0_E(Y)\not\subset I(s)$ and
\[
s\circ0_E=\id_Y\qquad\text{as rational maps}.
\]
The \emph{exceptional set} $Z(s)$ is the Zariski closed set \textup{(}\cref{rem:rational-basic}(1)\textup{)} of points $y\in Y$ such that $0_y\in I(s)$, or the differential of $s|_{E_y}$ at $0_y\in E_y$ is not surjective onto $T_yY$. The variety $Y$ is \emph{rationally elliptic} if it carries a rational spray $s$ with $Z(s)\ne Y$.
\end{definition}

Here rational ellipticity refers to rational sprays, not to the unrelated notion of a rationally elliptic space in rational homotopy theory.

\begin{definition}\label[definition]{def:alg-spray}
An \emph{algebraic spray} on a smooth variety $Y$ is a rational spray $s\colon E\to Y$ which is a morphism. Then $Z(s)=\{y\in Y:ds_{0_y}(E_y)\ne T_yY\}$ is its \emph{non-dominating locus}. The variety $Y$ is \emph{algebraically elliptic} if it carries an algebraic spray with $Z(s)=\emptyset$ \cite[Definition 5.6.13(e)]{F17}.
\end{definition}

Algebraically elliptic manifolds are Oka \cite[Corollary 5.6.14]{F17}. For a Zariski closed set $Z\subsetneq Y$ and $f_1\colon P\times\Op K\to Y\setminus Z$, let $(\mathrm H_Z)$ be condition (H) with the additional requirement that $f_t(P\times\Op K)\subset Y\setminus Z$ for all $t$.

\begin{theorem}\label[theorem]{thm:general}
Let $Y\subset\PP^N$ be a smooth projective variety with a rational spray $s\colon E\merto Y$, let $Z\subsetneq Y$ be a Zariski closed set containing $Z(s)$, let $X$, $G$, $K$, $X_0$, $P_0\subset P$, $\ell$ be as in \cref{thm:main2}, and let $f_1\colon P\times\Op K\to Y\setminus Z$ be a continuous family of holomorphic maps. Then the following hold.
\begin{enumerate}[(1)]
\item Conditions \textup{(A)} and $(\mathrm H_Z)$ are equivalent.
\item Under $(\mathrm H_Z)$, the conclusions \textup{(i)--(vii)} of \cref{thm:main2}\textup{(2)} hold.
\end{enumerate}
\end{theorem}

By (1), the condition $(\mathrm H_Z)$ is also necessary. For $Z=\emptyset$, \cref{thm:general} contains \cref{thm:main2}. Thus conclusions \textup{(1)} and \textup{(2)} of \cref{thm:main2} hold whenever $Y$ carries a rational spray $s$ with $Z(s)=\emptyset$. Algebraic ellipticity is used for \cref{thm:main1} and for the assertion following \cref{thm:main2}, through the Oka property of $Y$ and the nonvanishing of the conic lift in \cref{prop:normalized}(c).

The proof normalizes a rational spray to a degree-one conic lift (\cref{prop:normalized}), lifts the homotopy (\cref{lem:lifting}), approximates its lift (\cref{prop:runge-interp}) and substitutes back (\cref{prop:substitution}).

\begin{theorem}\label[theorem]{thm:ME}
Rational ellipticity is a birational invariant of smooth varieties, and
\[
\text{rational}\ \Longrightarrow\ \text{rationally elliptic}\ \Longrightarrow\ \text{unirational}.
\]
A smooth projective variety is rationally elliptic if and only if it carries an algebraic spray which is dominating at some point.
\end{theorem}

For smooth projective rational varieties, Banecki \cite[Corollary 1.9]{Ba24} proved algebraic ellipticity, extending the uniformly rational case \cite[Theorem 1.3]{AKZ24}, so \cref{thm:main1,thm:main2} apply to them directly.

The first implication in \cref{thm:ME} is not reversible: smooth cubic threefolds are nonrational \cite[Theorem 13.12]{CG72}, but are algebraically elliptic by \cite[Theorem 1.1]{KZ25}, and hence rationally elliptic. Very general smooth quartic double solids are algebraically elliptic and not even stably rational \cite[Corollary 1.2]{DKZ26}. These give nonrational smooth projective varieties $Y$ to which \cref{thm:main1,thm:main2} apply. The algebraic Oka-1 property is also a birational invariant of compact algebraic manifolds \cite[Corollary 3.4]{FL25}. Algebraic subellipticity and subellipticity of blow-ups are studied in \cite{LT17,KKT18}. For related results and the open birational questions, see \cite[\S4]{Za25}.

\section{Preliminaries}\label{sec:prelim}

\subsection{Notation}
Let $X$ be a complex space. For a holomorphic vector bundle $E\to X$ and an open set $U\subset X$ we write $\GO(U,E)$ for the holomorphic sections and $\cO(E)$ for the sheaf of holomorphic sections. For a holomorphic line bundle $L$ we write $L^k=L^{\otimes k}$.

\emph{Jets.} Let $Y$ be a complex manifold, $V\subset X$ an open set, $X_0\subset X$ a closed analytic subvariety with ideal sheaf $\cI_{X_0}$, and $\ell\ge1$ an integer. Holomorphic maps $g,h\colon V\to Y$ \emph{agree to order $\ell$ along $X_0$} if they are tangent to order $\ell$ along $X_0\cap V$ in the sense of \cite[Definition 1.3.2]{F17}, that is, $g=h$ on $X_0\cap V$ and, for local holomorphic coordinates $v$ on $Y$, the components of $v\circ g-v\circ h$ lie in $\cI_{X_0}^\ell$. We use the standard facts that this notion does not depend on the coordinates, is transitive, and is preserved under composition with holomorphic maps.

\emph{Projective spaces.} For $w\in\C^{N+1}\setminus\{0\}$ we write $[w]\in\PP^N$, so that $\cO_{\PP^N}(-1)_{[w]}=\C w$. For a projective variety $Y\subset\PP^N$ let $C(Y):=\{0\}\cup\{w\ne0:[w]\in Y\}$ be its affine cone and $\widehat Y:=C(Y)\setminus\{0\}$. If $L$ is a line bundle on $X$ and $\tilde f\in\GO(U,L)^{\oplus(N+1)}$, we write $\tilde f=u\,e$ in a local frame $e$ of $L$, with $u$ a $\C^{N+1}$-valued function. The condition $\tilde f(x)\ne0$ and the point $[\tilde f(x)]:=[u(x)]\in\PP^N$ do not depend on the frame. If $X$ is Stein and $\sigma\in\GO(X,L)$, then the open set $X\setminus\{\sigma=0\}$ is Stein \cite[Lemma 3.3]{Ab09}.

\subsection{Meromorphic maps and continuous families}
\begin{definition}\label[definition]{def:families}
Let $P$ be a topological space and $Y\subset\PP^N$ a smooth projective variety.
\begin{enumerate}[(a)]
\item A \emph{continuous family of holomorphic vector bundles} over $P\times X$ is a vector bundle with local trivializations over product open sets and transition functions continuous in $(p,x)$ and holomorphic in $x$. We write $E_p=E|_{\{p\}\times X}$ and $\GP(P'\times U,E)$ for continuous sections holomorphic in $x$. Such a family admits a continuous Hermitian metric when $P\times X$ is paracompact.
\item A \emph{continuous family of meromorphic maps} $F=[\widetilde F]\colon P\times X\merto Y$ consists of a continuous family of holomorphic line bundles $L$ and a tuple $\widetilde F\in\GP(P\times X,L)^{\oplus(N+1)}$ such that each $\widetilde F(p,\cdot)$ vanishes identically on no irreducible component of $X$ and $[\widetilde F(p,x)]\in Y$ whenever $\widetilde F(p,x)\ne0$. We write $L(F):=L$.

We say that $F$ is \emph{holomorphic on $P\times V$} if there is a continuous map $f\colon P\times V\to Y$ with holomorphic slices such that $f(p,x)=[\widetilde F(p,x)]$ whenever $\widetilde F(p,x)\ne0$. Requiring the tuple $\widetilde F$ to have no zeros there is stronger and will be stated explicitly. All approximation and homotopy conditions use tuples without zeros near $P\times K$.
\end{enumerate}
\end{definition}

Every continuous family of holomorphic maps $f\colon P\times X\to Y$, that is, every continuous map such that each $f(p,\cdot)$ is holomorphic, is a continuous family of meromorphic maps with $L:=f^*\cO_{\PP^N}(1)$ and $\widetilde F:=(f^*w_0,\dots,f^*w_N)$, where $w_0,\dots,w_N$ are the coordinate sections of $\cO_{\PP^N}(1)$. This tuple has no zeros.

\begin{lemma}\label[lemma]{lem:loc-triv}
Let $X$ be a Stein space, $P$ a Hausdorff space, $E\to P\times X$ a continuous family of holomorphic vector bundles, and $P'$ a contractible compact neighborhood of a point $p_0\in P$. Then $E|_{P'\times X}$ is isomorphic to $\pr_X^*E_{p_0}|_{P'\times X}$ as a continuous family of holomorphic vector bundles.
\end{lemma}

\begin{proof}
We may work componentwise in rank $r$. Let $\mathcal G$ be the group of invertible elements of $C(P',M_r(\C))$, namely $C(P',GL_r(\C))$. A jointly continuous function on $P'\times U$ with holomorphic slices defines a holomorphic $C(P',M_r(\C))$-valued map on $U$. Indeed, locally embed $U$ in a polydisc $D$ and use the lifting lemma in \cite[\S 6]{Bu68} (proved in the appendix of \cite{Ca58}) for the Fr\'echet surjection $\cO(D)\to\cO(U)$. The Cauchy formula then gives Banach-valued holomorphicity. The converse follows by evaluation.

For $x_0\in X$, contractibility gives a continuous frame over $P'\times\{x_0\}$. Extend its coefficients independently of $x$ in local frames and glue with a partition of unity in $p$. Compactness of $P'$ then gives a frame over $P'\times U$ for a neighborhood $U$ of $x_0$. Thus the transition functions define a holomorphic principal $\mathcal G$-bundle on $X$. The same holds for $\pr_X^*E_{p_0}$. Contractibility makes the two families topologically isomorphic. Compactness identifies this with an isomorphism of principal $\mathcal G$-bundles. Bungart's Oka--Grauert principle \cite[Theorem 8.1]{Bu68} makes it holomorphic, hence an isomorphism of continuous families.
\end{proof}

\subsection{Abe's hulls}
\begin{definition}[Abe \cite{Ab09}]\label[definition]{def:abe-hull}
Let $G\subset\Pic(X)$ be a \emph{submonoid}, i.e.\ a subsemigroup containing the trivial class $[\cO_X]$. For a compact set $K\subset X$ let $\cS(G,K)$ be the set of all $\sigma\in\GO(X,L)$, $L\in G$, with $\{\sigma=0\}\cap K=\emptyset$. The \emph{$G$-hull} of $K$ is
\[
\widetilde K_{X,G}:=\bigcap_{\sigma\in\cS(G,K)}\bigl(X\setminus\{\sigma=0\}\bigr),
\]
and $K$ is \emph{$G$-convex} if $\widetilde K_{X,G}=K$.
\end{definition}

Abe \cite{Ab09} works with subsemigroups. For a submonoid $G$, the $G$-hull of every compact set is compact \cite[Proposition 4.1, Corollary 4.5]{Ab09}, whereas for a subsemigroup it can be all of $X$ \cite[Example 4.2]{Ab09}. Adding $[\cO_X]$ to a subsemigroup can only shrink the $G$-hull, so $G$-convex sets remain $G$-convex. Every $\cO(X)$-convex compact set $K$ is $G$-convex for every submonoid $G$: if $x\notin K$, choose $h\in\cO(X)$ with $|h(x)|>\max_K|h|$. Then $\sigma:=h-h(x)\in\GO(X,\cO_X)$ vanishes at $x$ but not on $K$. For $G=\{[\cO_X]\}$, the $G$-hull is the meromorphically convex hull of \cite{Ab05}. If $X$ has no isolated points and $G=\Pic(X)$, it is the hull of $K$ with respect to hypersurfaces in the sense of \cite{Hi71} (see \cite[Proposition 2.2]{Ab09}). On a Stein manifold $X$, these two hulls agree for every compact set if and only if $\operatorname{Hom}(H_2(X;\Z),\Z)=0$ \cite[Theorem 1]{Co99}.

For $\sigma\in\GO(X,L)$ with $\{\sigma=0\}\cap K=\emptyset$ we write $\Khat_{X\setminus\{\sigma=0\}}$ for the $\cO(X\setminus\{\sigma=0\})$-convex hull of $K$ in the Stein space $X\setminus\{\sigma=0\}$ \cite[Lemma 3.3]{Ab09}. It is compact and has a basis of Stein open neighborhoods \cite[Proposition 2.5.1, Corollary 2.5.4]{F17}.

\begin{lemma}\label[lemma]{lem:hull}
Let $X$ be a Stein space, $G\subset\Pic(X)$ a submonoid, $K\subset X$ a $G$-convex compact set, $X_0\subset X$ a closed analytic subvariety and $V$ an open neighborhood of $K$. Then there are an integer $\mu\ge0$, $L_1,\dots,L_\mu\in G$ and $\sigma_j\in\GO(X,L_j)$ $(1\le j\le\mu)$ such that, with $A:=\bigcup_{j=1}^{\mu}\{\sigma_j=0\}$:
\begin{enumerate}[(a)]
\item $\{\sigma_j=0\}\cap K=\emptyset$, and $\sigma_j$ vanishes identically on no irreducible component of $X$ or of $X_0$, for every $j$,
\item $\Khat_{X\setminus A}\subset V$.
\end{enumerate}
If $\Khat_X\subset V$, for instance if $K$ is $\cO(X)$-convex, we choose $\mu=0$ and $A=\emptyset$. We put $L:=L_1\otimes\dots\otimes L_\mu\in G$ and $\sigma:=\sigma_1\cdots\sigma_\mu\in\GO(X,L)$, with $L:=\cO_X$ and $\sigma:=1$ if $\mu=0$, so that $\{\sigma=0\}=A$.
\end{lemma}

\begin{proof}
For closed analytic subvarieties $X''\subset X'\subset X$ disjoint from $K$, we have $\Khat_{X\setminus X'}\subset\Khat_{X\setminus X''}\subset\Khat_X$, since holomorphic functions restrict to smaller open sets. If $\Khat_X\subset V$, choose $\mu=0$. Hence we may assume otherwise.

We first observe the following. Let $L'\in\Pic(X)$, $\sigma_0\in\GO(X,L')$ with $\{\sigma_0=0\}\cap K=\emptyset$, and let $C\subset X$ be a compact set with $C\cap\Khat_{X\setminus\{\sigma_0=0\}}=\emptyset$. Then every $\sigma'\in\GO(X,L')$ sufficiently close to $\sigma_0$ satisfies $\{\sigma'=0\}\cap K=\emptyset$ and $C\cap\Khat_{X\setminus\{\sigma'=0\}}=\emptyset$. Indeed, otherwise there are $\sigma'_m\to\sigma_0$ and $x_m\in C\cap\Khat_{X\setminus\{\sigma'_m=0\}}$ converging to some $x\in C$. If $\sigma_0(x)=0$, then for $\tau\in\GO(X,L')$ with $\tau(x)\ne0$ the functions $\tau/\sigma'_m$ are bounded on $K$ uniformly in $m$ but unbounded at the points $x_m$. If $\sigma_0(x)\ne0$, then $x\notin\Khat_{X\setminus\{\sigma_0=0\}}$, and approximating a function separating $x$ from $K$ by $v/\sigma_0^{n}$ using \cite[Lemmas 3.1 and 3.2]{Ab09} gives the same contradiction with the functions $v/(\sigma'_m)^{n}$.

Let $\cS(G,K)$ be as in \cref{def:abe-hull} and put $K':=\Khat_X\setminus V$, a compact set. For $\sigma'\in\cS(G,K)$ the hull $\Khat_{X\setminus\{\sigma'=0\}}$ is compact, and since $K$ is $G$-convex,
\[
\bigcap_{\sigma'\in\cS(G,K)}\Khat_{X\setminus\{\sigma'=0\}}\subset\bigcap_{\sigma'\in\cS(G,K)}\bigl(X\setminus\{\sigma'=0\}\bigr)=\widetilde K_{X,G}=K\subset V.
\]
Hence the open sets $X\setminus\Khat_{X\setminus\{\sigma'=0\}}$, $\sigma'\in\cS(G,K)$, cover $K'$, and by compactness there are $\sigma_1,\dots,\sigma_\mu\in\cS(G,K)$, $\sigma_j\in\GO(X,L_j)$, and compact sets $C_j\subset X\setminus\Khat_{X\setminus\{\sigma_j=0\}}$ with $K'\subset C_1\cup\dots\cup C_\mu$.

For every irreducible component $X'$ of $X$ or of $X_0$, the set of sections of $L_j$ vanishing identically on $X'$ is a closed linear subspace of $\GO(X,L_j)$, proper by Cartan's Theorem A and hence nowhere dense. There are countably many such $X'$ because $X$ is second countable, and $\GO(X,L_j)$ is a Fr\'echet space for the topology of uniform convergence on compact sets. By the above observation (with $\sigma_0:=\sigma_j$ and $C:=C_j$) and Baire's theorem, we may replace each $\sigma_j$ by a nearby section which vanishes identically on no such $X'$, keeping $\{\sigma_j=0\}\cap K=\emptyset$ and $C_j\cap\Khat_{X\setminus\{\sigma_j=0\}}=\emptyset$. Then (a) holds, and since $\Khat_{X\setminus A}\subset\Khat_X\cap\Khat_{X\setminus\{\sigma_j=0\}}$ for every $j$, we get $\Khat_{X\setminus A}\subset\Khat_X\setminus(C_1\cup\dots\cup C_\mu)\subset\Khat_X\setminus K'\subset V$, which is (b).
\end{proof}

Consequently, every compact $G$-convex set $K$ has a basis of Stein open neighborhoods. Indeed, for a neighborhood $V$ of $K$, take $\sigma$ as in \cref{lem:hull}. Then $\widehat K_{X\setminus\{\sigma=0\}}\subset V$ has a Stein open neighborhood in the Stein space $X\setminus\{\sigma=0\}$ contained in $V$.

\section{Rational sprays}\label{sec:sprays}

\subsection{Rational sprays}
Let $s\colon E\merto Y$ be a rational spray on a smooth variety $Y$ (\cref{def:intro-spray}).

\begin{remark}\label[remark]{rem:rational-basic}
\begin{enumerate}[(1),wide=0pt]
\item $Z(s)$ is Zariski closed: the set of $y\in Y$ with $0_y\in I(s)$ is the preimage of $I(s)$ under the zero section, and on its complement the condition that $ds_{0_y}|_{E_y}$ is not surjective is Zariski closed.
\item \emph{Trivial bundles.} If $E$ is trivial over a Zariski open set $U$ and $\rho\colon Y\times\C^{r}\merto E$ is the induced rational map (an isomorphism over $U$), then $s_U:=s\circ\rho$ is a rational spray on the trivial bundle with $Z(s_U)\subset Z(s)\cup(Y\setminus U)$.
\end{enumerate}
\end{remark}

\subsection{Conic lifts}
From now on $Y\subset\PP^N$ is a smooth projective variety with affine cone $C(Y)$ and punctured cone $\widehat Y=C(Y)\setminus\{0\}$. For weights $m=(m_1,\dots,m_q)\in\Z_{\ge0}^q$ and ranks $r=(r_1,\dots,r_q)\in\Z_{\ge1}^q$ put
\[
\cO_Y(m;r):=\bigoplus_{j=1}^q\cO_Y(-m_j)^{\oplus r_j},\qquad |r|:=r_1+\dots+r_q.
\]
Note that the summands are the negative twists $\cO_Y(-m_j)$, whose fibers are $\cO_Y(-m_j)_{[w]}=\C\,w^{\otimes m_j}\subset(\C^{N+1})^{\otimes m_j}$. For $\lambda\in\C^*$ and $\zeta=(\zeta_1,\ldots,\zeta_q)\in\C^{r_1}\times\cdots\times\C^{r_q}=\C^{|r|}$ put $\lambda^{-m}\cdot\zeta:=(\lambda^{-m_1}\zeta_1,\dots,\lambda^{-m_q}\zeta_q)$. For such $\zeta$ and $w\in\widehat Y$, let
\[
\zeta\cdot w^{\otimes m}:=(\zeta_j\cdot w^{\otimes m_j})_j\in\cO_Y(m;r)_{[w]},
\]
where $\zeta_j\cdot w^{\otimes m_j}:=(\zeta_{j,1}w^{\otimes m_j},\ldots,\zeta_{j,r_j}w^{\otimes m_j})$. This identifies $\cO_Y(m;r)$ with $(\widehat Y\times\C^{|r|})/\C^*$ for the action $\lambda\cdot(w,\zeta)=(\lambda w,\lambda^{-m}\cdot\zeta)$.

For every algebraic vector bundle $E$ there is a surjective morphism of vector bundles $\varpi\colon\cO_Y(-m)^{\oplus r}\to E$ for $m\gg0$, since $E\otimes\cO_Y(m)$ is generated by global sections \cite[n$^\circ$\,66, Th\'eor\`eme 2(a)]{Se55}, and for a rational spray $s\colon E\merto Y$ the composition $s\circ \varpi\colon\cO_Y(-m)^{\oplus r}\merto Y$ is a rational spray with $Z(s\circ \varpi)=Z(s)$. Indeed, a local splitting of $\varpi$ shows that regularity of $s\circ \varpi$ at the zero section implies regularity of $s$ there, and surjectivity of $\varpi$ preserves domination. Hence it suffices to consider the bundles $\cO_Y(m;r)$. Throughout this section $H^0$ denotes algebraic sections.

\begin{definition}\label[definition]{def:conic-lift}
Let $s\colon\cO_Y(m;r)\merto Y$ be a rational spray. A \emph{conic lift} of $s$ of degree $d\ge1$ is a morphism of affine varieties $R\colon C(Y)\times\C^{|r|}\to C(Y)$ (i.e.\ the restriction of a polynomial map $\C^{N+1}\times\C^{|r|}\to\C^{N+1}$ mapping $C(Y)\times\C^{|r|}$ into $C(Y)$) such that
\[
R(\lambda\cdot(w,\zeta))=\lambda^dR(w,\zeta)\quad(\lambda\in\C^*),
\]
and $s\circ\pi=[R]$ as rational maps, where $\pi\colon\widehat Y\times\C^{|r|}\to\cO_Y(m;r)$, $(w,\zeta)\mapsto\zeta\cdot w^{\otimes m}$, is the quotient map. Since $[R]$ is a $\C^*$-invariant morphism on $\{R\ne0\}$, $s$ is then regular at $\zeta\cdot w^{\otimes m}$ with $s(\zeta\cdot w^{\otimes m})=[R(w,\zeta)]$ whenever $(w,\zeta)\in\widehat Y\times\C^{|r|}$ and $R(w,\zeta)\ne0$. We write $R=\sum_{k=0}^{\nu}R_k$, where $R_k=(R_{k,0},\dots,R_{k,N})$ is homogeneous of degree $k$ in $\zeta$ and $\nu:=\deg_\zeta R$. Then $R_k(\lambda\cdot(w,\zeta))=\lambda^dR_k(w,\zeta)$, and $R_0(w)=R(w,0)$. Put
\[
Z(R):=Z(s)\cup\{[w]\in Y:R_0(w)=0\}.
\]
\end{definition}

\begin{remark}\label[remark]{rem:spray-basic}
\begin{enumerate}[(1),wide=0pt]
\item \emph{The degree-$0$ part.} Since $s(0_{[w]})=[w]$ for generic $[w]$ and the condition $R_0(w)\in\C w$ is closed, we have $R_0(w)=\gamma(w)w$ on $\widehat Y$, where $\gamma$ is a regular function on $\widehat Y$, homogeneous of degree $d-1$, i.e.\ a section of $\cO_Y(d-1)$. If $d=1$ and $R_0\not\equiv0$, then $\gamma$ is a nonzero constant and we normalize $R_0(w)=w$. Then $Z(R)=Z(s)$ and $s$ is regular on a neighborhood of the zero section.
\item \emph{Homogenization.} Let $R^h(w,\zeta,t):=\sum_{k=0}^{\nu}t^{\nu-k}R_k(w,\zeta)$. For $t\ne0$, $R^h(w,\zeta,t)=t^{\nu}R(w,\zeta/t)$, so $R^h(C(Y)\times\C^{|r|}\times\C)\subset C(Y)$ by closedness of $C(Y)$.
\item If $R_0(w)\ne0$, then $[R]$ is a morphism near $(w,0)$ inducing $s$, so $s$ is regular near $0_y$ and $s(0_y)=y$, where $y:=[w]$. The set $Z(R)$ is Zariski closed by its definition and \cref{rem:rational-basic}(1), $Y\setminus Z(R)$ is a complex manifold, and
\[
\Omega:=\{\zeta\cdot w^{\otimes m}:[w]\in Y\setminus Z(R),\ R(w,\zeta)\ne0,\ [R(w,\zeta)]\in Y\setminus Z(R)\}
\]
is an open neighborhood of the zero section of $\cO_Y(m;r)|_{Y\setminus Z(R)}$ on which $s$ is holomorphic with values in $Y\setminus Z(R)$. It is a \emph{local dominating spray}: $s(0_y)=y$ and $ds_{0_y}(\cO_Y(m;r)_y)=T_yY$ for $y\in Y\setminus Z(R)$.
\end{enumerate}
\end{remark}

\subsection{Existence of conic lifts}
\begin{lemma}\label[lemma]{lem:normal-form}
Let $Y\subset\PP^N$ be a smooth projective variety and $s\colon E\merto Y$ a rational spray on an algebraic vector bundle $p\colon E\to Y$. There are a line bundle $L\to Y$ and algebraic sections $\tilde s=(\tilde s_0,\dots,\tilde s_N)\in H^0(E,p^*L)^{\oplus(N+1)}$ such that:
\begin{enumerate}[(a)]
\item $s=[\tilde s_0:\dots:\tilde s_N]$, and $s$ is regular at $v\in E$ if and only if $\tilde s(v)\ne0$,
\item $\tilde s_i=\sum_{k=0}^{\nu}\tilde s_{k,i}$, where $\tilde s_{k,i}\in H^0(Y,L\otimes\Sym^kE^*)$ is regarded as a section of $p^*L$ which is homogeneous of degree $k$ on the fibers of $E$, and there is $\gamma\in H^0(Y,L\otimes\cO_Y(-1))$ such that $\tilde s_{0,i}=\gamma\otimes w_i$ for $0\le i\le N$, where $w_0,\dots,w_N\in H^0(Y,\cO_Y(1))$ are the coordinate sections,
\item $\{\gamma=0\}\subset Z(s)$. If $s$ is regular along the zero section, then $\gamma$ has no zeros, so that $L\cong\cO_Y(1)$, and under this isomorphism $\tilde s_{0,i}=w_i$.
\end{enumerate}
\end{lemma}

\begin{proof}
Since $E$ is smooth and $Y$ is projective, $I(s)$ has codimension at least two. The bundle $s^*\cO_Y(1)$ on $E\setminus I(s)$ extends to $E$, as do its coordinate sections. Their common zero set is $I(s)$, since it does not meet $E\setminus I(s)$, and $s=[\tilde s_0:\dots:\tilde s_N]$ is regular wherever some $\tilde s_i$ does not vanish. Since $p^*\colon\Pic(Y)\to\Pic(E)$ is an isomorphism \cite[Theorem 5]{Ma75}, the extended bundle is $p^*L$, proving (a).

The identity $p_*\cO_E^{\mathrm{alg}}=\bigoplus_{k\ge0}\Sym^kE^*$ gives the finite homogeneous decomposition in (b). On the zero section, $[\tilde s(0_y)]=y$ generically. The closed condition $\tilde s(0_y)\in L_y\otimes\cO_Y(-1)_y$ therefore holds everywhere, giving $\gamma$ and $\tilde s_{0,i}=\gamma\otimes w_i$. By (a), $\gamma(y)=0$ precisely where $s$ is indeterminate at $0_y$, which proves (c).
\end{proof}

\begin{proposition}\label[proposition]{prop:normalized}
Let $Y\subset\PP^N$ be a smooth projective variety, $s\colon E\merto Y$ a rational spray on $p\colon E\to Y$, and let $L$, $\tilde s$ and $\gamma$ be as in \cref{lem:normal-form}. Put $E':=E\otimes L^{-1}\otimes\cO_Y(1)$. We also write $p$ for the projection $E'\to Y$.
\begin{enumerate}[(a)]
\item For $v'\in E'_y$ we have $\gamma(y)\,v'\in E_y$, and $s'(v'):=s(\gamma\,v')$ defines a rational spray $s'\colon E'\merto Y$ with $s'=[\tilde s'_0:\dots:\tilde s'_N]$, where
\[
\tilde s'_i:=w_i+\sum_{k\ge1}\gamma^{k-1}\tilde s_{k,i}\in H^0\bigl(E',p^*\cO_Y(1)\bigr)\qquad(0\le i\le N).
\]
In particular $s'$ is regular on a neighborhood of the zero section of $E'$, and $Z(s')\subset Z(s)$. If $s$ is regular along the zero section, then $L\otimes\cO_Y(-1)$ is trivialized by $\gamma$, $E'=E$ and $s'=s$.
\item For $m\gg0$ and a surjective morphism of vector bundles $\varpi\colon\cO_Y(-m)^{\oplus r'}\to E'$, the rational spray $s'\circ \varpi$ has a conic lift $R$ of weight $(m)$ and degree $1$ with $R_0(w)=w$. In particular $Z(R)=Z(s'\circ \varpi)\subset Z(s)$.
\item If $s$ is an algebraic spray, then $s'=s$ and the lift in \textup{(b)} satisfies $R(w,\zeta)\ne0$ on $\widehat Y\times\C^{r'}$. In particular every algebraically elliptic smooth projective variety has a rational spray with a conic lift $R$ of degree $1$ such that $Z(R)=\emptyset$ and $R$ has no zeros on $\widehat Y\times\C^{r'}$.
\end{enumerate}
\end{proposition}

\begin{proof}
(a) Since $E'\otimes(L\otimes\cO_Y(-1))=E$, we have $\gamma(y)v'\in E_y$. Here $\gamma^{k-1}\tilde s_{k,i}$ is evaluated on $v'^{\otimes k}$. It is a section of $(L\otimes\cO_Y(-1))^{k-1}\otimes L\otimes(L^{-1}\otimes\cO_Y(1))^{k}=\cO_Y(1)$. By \cref{lem:normal-form}(b),
\[
\tilde s_i(\gamma v')=\sum_k\gamma^k\tilde s_{k,i}(v'^{\otimes k})=\gamma\Bigl(w_i+\sum_{k\ge1}\gamma^{k-1}\tilde s_{k,i}(v'^{\otimes k})\Bigr)=\gamma\,\tilde s'_i(v').
\]
Hence $[\tilde s'_0(v'):\dots:\tilde s'_N(v')]=s(\gamma(y)v')=s'(v')$ for $v'\in E'_y$ with $\gamma(y)\ne0$, so $\tilde s'=(\tilde s'_0,\dots,\tilde s'_N)$ represents $s'$. In particular $s'$ is regular at every $v'\in E'$ with $\tilde s'(v')\ne0$, for instance near the zero section, where $\tilde s'_i=w_i$. For $y\notin Z(s)$ we have $\gamma(y)\ne0$, and $ds'_{0_y}|_{E'_y}=ds_{0_y}\circ\gamma(y)$ is surjective, so $Z(s')\subset Z(s)$. If $s$ is regular along the zero section, \cref{lem:normal-form}(c) gives the last statement.

(b) For $m\gg0$, $E'\otimes\cO_Y(m)$ is generated by global sections \cite[n$^\circ$\,66, Th\'eor\`eme 2(a)]{Se55}, and $H^0(\PP^N,\cO(n))\to H^0(Y,\cO_Y(n))$ is onto for $n\ge m+1$, since $H^1(\PP^N,\cI_Y^{\mathrm{alg}}(n))=0$ for $n\gg0$ \cite[n$^\circ$\,66, Th\'eor\`eme 2(b)]{Se55}, where $\cI_Y^{\mathrm{alg}}$ is the ideal sheaf of $Y$ in $\PP^N$. Choose such $m$ and $\varpi$. By the description of the regular functions on the total space of a vector bundle in the proof of \cref{lem:normal-form}, applied to $\cO_Y(-m)^{\oplus r'}$, we have $\tilde s'_i\circ \varpi=\sum_ka_{k,i}$ with $a_{k,i}\in H^0\bigl(Y,\cO_Y(1)\otimes\Sym^k(\cO_Y(m)^{\oplus r'})\bigr)$ and $a_{0,i}=w_i$. That is, $a_{k,i}$ is a homogeneous polynomial of degree $k$ in $\zeta\in\C^{r'}$ whose coefficients are sections of $\cO_Y(1+km)$. For $k\ge1$ these coefficients lift by the surjectivity above to homogeneous polynomials of degree $1+km$ on $\C^{N+1}$, which gives polynomials $R_{k,i}(w,\zeta)$ of bidegree $(1+km,k)$. Put $R_{0,i}:=w_i$, $R_k:=(R_{k,0},\dots,R_{k,N})$ and $R:=\sum_kR_k$. By construction, pairing $(\tilde s'\circ \varpi)(\zeta\cdot w^{\otimes m})\in\cO_Y(1)_{[w]}\otimes\C^{N+1}$ with $w\in\cO_Y(-1)_{[w]}$ gives $R(w,\zeta)$. Hence $(s'\circ \varpi)(\zeta\cdot w^{\otimes m})=[R(w,\zeta)]\in Y$ where $R(w,\zeta)\ne0$. Since moreover $R(0,\zeta)=0$, $R$ maps $C(Y)\times\C^{r'}$ into $C(Y)$. Thus $R$ is a conic lift of degree $1$ with $R_0(w)=w$, and $Z(R)=Z(s'\circ \varpi)$ by \cref{rem:spray-basic}(1). Since $\varpi$ is surjective, $Z(s'\circ \varpi)=Z(s')\subset Z(s)$.

(c) An algebraic spray is regular along the zero section, so $s'=s$ by (a). Since $s\circ \varpi$ is a morphism, the components of $\tilde s\circ \varpi$ have no common zero, so $R$ has no zeros on $\widehat Y\times\C^{r'}$. An algebraically elliptic variety carries an algebraic spray with $Z(s)=\emptyset$.
\end{proof}

\subsection{Homogeneous manifolds}
\begin{example}\label[example]{ex:homogeneous}
\begin{enumerate}[(a),wide=0pt]
\item For $\PP^N$, let $J:=\{(i,j):0\le i\ne j\le N\}$ and $A(\zeta):=\prod_{(i,j)\in J}(I+\zeta_{ij}E_{ij})$ \textup{(}fixed order, $E_{ij}$ the matrix units\textup{)}. Then $s([w],\zeta):=[A(\zeta)w]$ is an algebraic spray on $\PP^N\times\C^J$ with $Z(s)=\emptyset$, and $R(w,\zeta):=A(\zeta)w$ is a conic lift of weight $0$ and degree $1$.
\item Let $G'\subset GL_{N+1}(\C)$ be a connected semisimple linear algebraic group with Lie algebra $\mathfrak g$, and $Y\subset\PP^N$ a closed $G'$-orbit. Choose nilpotent $\theta_1,\dots,\theta_M\in\mathfrak g$ spanning $\mathfrak g$ and put $A(\zeta):=\prod_i\exp(\zeta_i\theta_i)$. Then $s(y,\zeta):=A(\zeta)\cdot y$ is an algebraic spray on $Y\times\C^M$ with $Z(s)=\emptyset$, and $R(w,\zeta):=A(\zeta)w$ is a conic lift of weight $0$ and degree $1$ such that $R$ has no zeros on $\widehat Y\times\C^M$.
\end{enumerate}
\end{example}

\begin{proof}
(a) $A(\zeta)$ is invertible, polynomial in $\zeta$, and $A(0)=I$. If $w_j\ne0$, the vectors $E_{ij}w=w_je_i$ ($i\ne j$) span $\C^{N+1}/\C w\cong T_{[w]}\PP^N$, so $Z(s)=\emptyset$.
(b) The span of the nilpotent cone is an ideal containing all root spaces, hence equals $\mathfrak g$. Since $\theta_i$ is nilpotent, $A(\zeta)$ is polynomial in $\zeta$, and $A(\zeta)\in G'$ preserves $C(Y)$. The differential of $\zeta\mapsto A(\zeta)\cdot y$ at $0$ has image the image of $\mathfrak g\to T_yY$, which is onto since $Y$ is homogeneous.
\end{proof}

\subsection{Iteration}
Let $Y\subset\PP^N$ be a smooth projective variety, $R=\sum_kR_k$ a conic lift of degree $d$ of a rational spray $s$ on $\cO_Y(m;r)$, and assume $Z(R)\ne Y$. Define inductively $R^{(0)}(w):=w$, $R^{(1)}:=R$ and
\[
R^{(l)}(w;\zeta_1,\dots,\zeta_l):=R\bigl(R^{(l-1)}(w;\zeta_1,\dots,\zeta_{l-1}),\zeta_l\bigr),\qquad\zeta_i\in\C^{|r|},
\]
and write $R^{(l)}_0(w):=R^{(l)}(w;0)$ for its degree-$0$ part. Put $m^{(l)}:=(m,dm,\dots,d^{l-1}m)$, $r^{(l)}:=(r,\dots,r)$ ($l$ times) and
\[
E^{(l)}:=\cO_Y(m^{(l)};r^{(l)})=\cO_Y(m;r)\oplus\cO_Y(dm;r)\oplus\dots\oplus\cO_Y(d^{l-1}m;r),
\]
whose points we write $\zeta\cdot w^{\otimes m^{(l)}}$, $\zeta=(\zeta_1,\dots,\zeta_l)$. The corresponding $\C^*$-action on $\widehat Y\times\C^{l|r|}$ is $\lambda\cdot(w,\zeta)=(\lambda w,\lambda^{-m}\cdot\zeta_1,\lambda^{-dm}\cdot\zeta_2,\dots,\lambda^{-d^{l-1}m}\cdot\zeta_l)$. Put
\begin{multline*}
\Omega^{(l)}:=\Bigl\{\zeta\cdot w^{\otimes m^{(l)}}:(w,\zeta)\in\widehat Y\times\C^{l|r|},\ R^{(i)}(w;\zeta_1,\dots,\zeta_i)\ne0\\
\text{and }\bigl[R^{(i)}(w;\zeta_1,\dots,\zeta_i)\bigr]\in Y\setminus Z(R)\text{ for }0\le i\le l\Bigr\}.
\end{multline*}

\begin{lemma}\label[lemma]{lem:iteration}
\begin{enumerate}[(a)]
\item $R^{(l)}(\lambda\cdot(w,\zeta))=\lambda^{d^l}R^{(l)}(w,\zeta)$ for $\lambda\in\C^*$, $R^{(l)}(C(Y)\times\C^{l|r|})\subset C(Y)$, and $R^{(l)}_0=R_0^{\circ l}$.
\item There is a unique rational spray $s^{(l)}\colon E^{(l)}\merto Y$ such that $R^{(l)}$ is a conic lift of degree $d^l$ of $s^{(l)}$.
\item $\Omega^{(l)}$ is an open neighborhood of the zero section of $E^{(l)}|_{Y\setminus Z(R)}$, and $s^{(l)}\colon\Omega^{(l)}\to Y\setminus Z(R)$ is holomorphic. For $\zeta\cdot w^{\otimes m^{(l)}}\in\Omega^{(l)}$ and $l\ge2$,
\begin{equation}\label{eq:iterate}
s^{(l)}(\zeta\cdot w^{\otimes m^{(l)}})=s\bigl(\zeta_l\cdot R^{(l-1)}(w;\zeta_1,\dots,\zeta_{l-1})^{\otimes m}\bigr).
\end{equation}
\end{enumerate}
\end{lemma}

\begin{proof}
Equivariance, the cone condition and $R^{(l)}_0=R_0^{\circ l}$ follow by induction. By \cref{rem:spray-basic}(1), $R_0(w)=\gamma(w)w$ with $\gamma(w)\ne0$ over $Y\setminus Z(R)$, so $R^{(i)}_0(w)\in\C^*w$ there. On the invariant open set $\{R^{(l)}\ne0\}$, the morphism $[R^{(l)}]$ descends along the principal $\C^*$-bundle $\widehat Y\times\C^{l|r|}\to E^{(l)}$. It defines the unique rational map $s^{(l)}$, and the preceding observation gives $s^{(l)}(0_y)=y$ generically. This proves (a) and (b).

The defining conditions of $\Omega^{(l)}$ are open, invariant and satisfied on the zero section over $Y\setminus Z(R)$. On this set $s^{(l)}=[R^{(l)}]$ is holomorphic. The recursion for $R^{(l)}$, followed by projectivization, gives \eqref{eq:iterate}.
\end{proof}

For weight zero, $s^{(l)}$ is the ordinary iterated composition of $s$. For positive weights, \eqref{eq:iterate} uses the cone representative of each intermediate image to identify the next fiber.

\subsection{Substitution}
Let $L_1$ be a holomorphic line bundle on a complex space $X$, $x\in X$, and $\tilde f\in L_{1,x}\otimes\C^{N+1}$ with $\tilde f\ne0$ and $[\tilde f]\in Y$. For $\kappa=(\kappa_j)_j$ with $\kappa_j=(\kappa_{j,1},\dots,\kappa_{j,r_j})\in(L_{1,x}^{-m_j})^{\oplus r_j}$ we put
\[
\kappa\otimes\tilde f^{\otimes m}:=\bigl(\langle\kappa_{j,\mu},\tilde f^{\otimes m_j}\rangle\bigr)_{j,\mu},
\]
where $\tilde f^{\otimes m_j}\in L_{1,x}^{m_j}\otimes(\C^{N+1})^{\otimes m_j}$ and $\langle\cdot,\cdot\rangle$ is the contraction $L_{1,x}^{-m_j}\otimes L_{1,x}^{m_j}\to\C$. If $e$ is a basis of $L_{1,x}$, $\tilde f=e\otimes u$ and $\kappa_{j,\mu}=\hat\kappa_{j,\mu}(e^*)^{\otimes m_j}$, then $\kappa\otimes\tilde f^{\otimes m}=\hat\kappa\cdot u^{\otimes m}$. In particular $\kappa\otimes\tilde f^{\otimes m}\in\cO_Y(m;r)_{[\tilde f]}$.

\begin{proposition}\label[proposition]{prop:substitution}
Let $Y\subset\PP^N$ be a smooth projective variety and $R=\sum_{k=0}^{\nu}R_k$ a conic lift of degree $d$ of a rational spray $s\colon\cO_Y(m;r)\merto Y$. Let $X$ be a complex space, $L_1$ and $L_2$ holomorphic line bundles on $X$, $\tilde f\in\GO(X,L_1)^{\oplus(N+1)}$ such that $[\tilde f(x)]\in Y$ whenever $\tilde f(x)\ne0$, and put $f:=[\tilde f]$. Let $\sigma\in\GO(X,L_2)$, $n\in\Z_{\ge0}$, and $\eta=(\eta_j)$ with $\eta_j\in\GO(X,L_1^{-m_j}\otimes L_2^n)^{\oplus r_j}$. We define the section $\widetilde F$ by
\[
\widetilde F:=R^h(\tilde f,\eta,\sigma^n)\in\GO\bigl(X,L_1^d\otimes L_2^{\nu n}\bigr)^{\oplus(N+1)}.
\]
\begin{enumerate}[(i)]
\item $[\widetilde F(x)]\in Y$ whenever $\widetilde F(x)\ne0$.
\item On $X\setminus\{\sigma=0\}$ we have $\widetilde F=\sigma^{\nu n}R(\tilde f,\eta/\sigma^n)$. Hence, if $\sigma(x)\ne0$, $\tilde f(x)\ne0$ and $R(\tilde f(x),(\eta/\sigma^n)(x))\ne0$, then $\widetilde F(x)\ne0$ and $[\widetilde F(x)]=s\bigl((\eta/\sigma^n)(x)\otimes\tilde f(x)^{\otimes m}\bigr)$.
\item If $\tilde f$ has no zeros on an open set $U\subset X$, then $\kappa\mapsto\kappa\otimes\tilde f^{\otimes m}$ is an isomorphism from $\bigoplus_j(L_1^{-m_j})^{\oplus r_j}|_U$ onto $f^*\cO_Y(m;r)|_U$.
\item The same holds for continuous families over $P\times X$, where $P$ is a topological space \textup{(}$\sigma$ independent of $p$\textup{)}. Then $\widetilde F\in\GP(P\times X,L_1^d\otimes\pr_X^*L_2^{\nu n})^{\oplus(N+1)}$.
\end{enumerate}
\end{proposition}

\begin{proof}
In local frames $e_1$ of $L_1$ and $e_2$ of $L_2$ write $\tilde f=u\,e_1$, $\eta_j=v_j\,e_1^{-m_j}e_2^n$ and $\sigma=c\,e_2$. The defining ideal of $C(Y)$ is homogeneous, so equivariance allows us to choose polynomial representatives of the coordinates of $R$ that are homogeneous for the weighted action. For such representatives every monomial $w^\alpha\zeta^\beta t^{\nu-k}$ of $R^h$ satisfies $|\alpha|-\sum_jm_j|\beta_j|=d$ and $|\beta|+\nu-k=\nu$. Hence $R^h(u,v,c^n)\,e_1^de_2^{\nu n}$ does not depend on the frames, and it equals $\widetilde F$. (i) follows from the assumption and \cref{rem:spray-basic}(2). At points $x$ with $c(x)\ne0$, we have $R^h(u,v,c^n)=c^{\nu n}R(u,v/c^n)$ and $\eta/\sigma^n=(v/c^n)e_1^{-m}$, which gives the identity in (ii). For (iii), $f=[\tilde f]\colon U\to Y$ is holomorphic. Regarding $\tilde f$ as a section of $\mathrm{Hom}(L_1^{-1},\C^{N+1})$, contraction with $\tilde f$ gives a holomorphic vector bundle map $\iota_{\tilde f}\colon L_1^{-1}|_U\to U\times\C^{N+1}$, $\lambda\mapsto\langle\lambda,\tilde f\rangle$. It is injective on fibers since $\tilde f$ has no zeros on $U$, and its image is $f^*\cO_Y(-1)$. Hence $\iota_{\tilde f}\colon L_1^{-1}|_U\to f^*\cO_Y(-1)$ is an isomorphism, and $\kappa\mapsto\kappa\otimes\tilde f^{\otimes m}$ is the isomorphism $\bigoplus_j(\iota_{\tilde f}^{\otimes m_j})^{\oplus r_j}$. (iv) is clear.
\end{proof}

\section{A meromorphic Oka--Weil theorem with parameters}\label{sec:runge}

\begin{proposition}\label[proposition]{prop:runge-interp}
Let $X$ be a Stein space, $P_0\subset P$ compact Hausdorff spaces such that every point of $P$ has a contractible compact neighborhood, $E\to P\times X$ a continuous family of holomorphic vector bundles with a continuous Hermitian metric, $L\to X$ a holomorphic line bundle, $\sigma\in\GO(X,L)$ a section which vanishes identically on no irreducible component of $X$, and $\cI\subset\cO_X$ a coherent ideal sheaf. Assume that $K\subset X\setminus\{\sigma=0\}$ is a compact $\cO(X\setminus\{\sigma=0\})$-convex set, and that $f\in\GP(P\times\Op K,E)$ satisfies
\begin{itemize}
\item $f(p,\cdot)\in\Gamma(\Op K,\cI\cdot\cO(E_p))$ for all $p\in P$, and
\item $f(p,\cdot)=0$ for $p\in P_0$.
\end{itemize}
Then for every $\epsilon>0$ there are $n\in\Z_{\ge0}$ and $\eta\in\GP(P\times X,E\otimes\pr_X^*L^n)$ such that for all $p\in P$:
\begin{enumerate}[label=(\arabic*)]
\item $\eta(p,\cdot)\in\Gamma(X,\cI\cdot\cO(E_p\otimes L^n))$,
\item $\eta(p,\cdot)=0$ if $p\in P_0$,
\item $\sup_{P\times K}\|f-\eta/\sigma^n\|<\epsilon$.
\end{enumerate}
\end{proposition}

\begin{proof}
It suffices to show that $f$ can be approximated uniformly on $P\times K$, as closely as we like, by $\eta/\sigma^n$ with $\eta$ satisfying (1) and (2).

\emph{Step 1.} Assume that $P$ is a singleton, and write $f$ for $f(p,\cdot)$. Then $f$ is a section of the coherent sheaf $\cF:=\cI\cdot\cO(E)$ over $\Op K$. Since $X\setminus\{\sigma=0\}$ is Stein \cite[Lemma 3.3]{Ab09} and $K$ is $\cO(X\setminus\{\sigma=0\})$-convex, we may assume, shrinking $\Op K$, that $\Op K$ is Stein. By Cartan's Theorem A \cite[Theorem 2.6.1]{F17} and compactness there are $g_1,\dots,g_q\in\Gamma(X,\cF)$ generating the stalks of $\cF$ at all points of a neighborhood of $K$, and shrinking $\Op K$ again we may assume that they generate $\cF$ at all points of $\Op K$. By Cartan's division theorem \cite[Corollary 2.6.4]{F17}, $f=\sum_j\xi_jg_j$ with $\xi_j\in\cO(\Op K)$. The Oka--Weil theorem \cite[Theorem 2.3.1]{F17} on $X\setminus\{\sigma=0\}$ gives $\xi'_j\in\cO(X\setminus\{\sigma=0\})$ approximating $\xi_j$ uniformly on $K$. By \cite[Lemma 3.1]{Ab09}, each $\xi'_j$ is in turn approximated uniformly on $K$ by $u_j/\sigma^n$ with $u_j\in\GO(X,L^n)$. A common $n$ is obtained by multiplying by powers of $\sigma$. Then $\eta:=\sum_ju_jg_j$ satisfies (1), and $\eta/\sigma^n=\sum_j(u_j/\sigma^n)g_j$ approximates $f$ uniformly on $K$ as closely as we like.

\emph{Step 2: constant bundles.} Suppose $E=\pr_X^*E_{p'}$. Since $f$ is continuous and $K$ is compact, there are finitely many $p_j\in P$ and a partition of unity $(\psi_j)$ on $P$ such that $f(p,\cdot)$ is uniformly close to $f(p_j,\cdot)$ on $K$ whenever $\psi_j(p)\ne0$. Apply Step 1 to each $f(p_j,\cdot)$, with a common exponent $n$, and let $\eta_j$ be the resulting sections. Then $\eta'(p,\cdot):=\sum_j\psi_j(p)\eta_j$ satisfies (1), and $\eta'(p,\cdot)/\sigma^n$, being a convex combination of approximations, is uniformly close to $f(p,\cdot)$ on $K$ for every $p\in P$. Since $f(p,\cdot)=0$ for $p\in P_0$, there is a neighborhood $W$ of $P_0$ such that $f(p,\cdot)$ is uniformly small on $K$ for $p\in W$. Choose a continuous function $\chi\colon P\to[0,1]$ equal to $1$ on $P_0$ with support in $W$, and put $\eta:=(1-\chi)\eta'$. Then $\eta$ satisfies (1) and (2), and
\[
f(p,\cdot)-\eta(p,\cdot)/\sigma^n=(1-\chi(p))\bigl(f(p,\cdot)-\eta'(p,\cdot)/\sigma^n\bigr)+\chi(p)f(p,\cdot)
\]
is uniformly small on $K$ for every $p\in P$.

\emph{Step 3: localization in $P$.} By \cref{lem:loc-triv}, finitely many contractible compact neighborhoods $P_i$ have interiors covering $P$ and admit isomorphisms $\iota_i\colon\pr_X^*E_{q_i}\to E$ over $P_i\times X$. Choose a partition of unity $(\varphi_i)$ with $\supp\varphi_i\subset P_i^\circ$. Apply Step 2 to $\iota_i^{-1}f$ over $P_i$, relative to $P_0\cap P_i$, with a common exponent $n$, and let $\eta_i$ be the resulting sections. Then $\eta:=\sum_i\varphi_i\,\iota_i(\eta_i)$ satisfies (1) and (2), and
\[
f(p,\cdot)-\eta(p,\cdot)/\sigma^n=\sum_i\varphi_i(p)\,\iota_i\bigl(\iota_i^{-1}f(p,\cdot)-\eta_i(p,\cdot)/\sigma^n\bigr).
\]
Since the isomorphisms $\iota_i$ are uniformly bounded on $P_i\times K$, this is uniformly small on $K$ for every $p\in P$ if the approximations in Step 2 are sufficiently close.
\end{proof}

The local contractibility assumption on $P$ enters only in Step 3 through \cref{lem:loc-triv}. For bundles pulled back from $X$, Step 2 works for arbitrary compact Hausdorff parameter pairs.

\section{Lifting homotopies through rational sprays}\label{sec:lifting}

\begin{setting}\label[setting]{set:spray}
In this section $Y\subset\PP^N$ is a smooth projective variety, $R$ is a conic lift of degree $d$ of a rational spray $s$ on $E=\cO_Y(m;r)$ (\cref{def:conic-lift}), $Z\subsetneq Y$ is a Zariski closed set containing $Z(R)$, and $\Omega$, $E^{(l)}$, $\Omega^{(l)}$, $s^{(l)}$ are as in \cref{rem:spray-basic} and \cref{lem:iteration}. In their definitions we replace $Z(R)$ by $Z$. This restricts the domains so that every intermediate value lies in $Y\setminus Z$, and the same proofs apply. Thus $s\colon\Omega\to Y\setminus Z$ is a holomorphic local dominating spray on the open neighborhood $\Omega$ of the zero section of $E|_{Y\setminus Z}$. We fix Hermitian metrics on $E|_{Y\setminus Z}$ and $T(Y\setminus Z)$. For a continuous family of holomorphic line bundles $L$ over $P\times X$ we put $E^{(l)}(L):=\bigoplus_{i=1}^l\bigoplus_j(L^{-d^{i-1}m_j})^{\oplus r_j}$. For $\kappa\in E^{(l)}(L)_{(p,x)}$ and $\tilde f\in L_{(p,x)}\otimes\C^{N+1}$ with $\tilde f\ne0$ and $[\tilde f]\in Y$ we write $\kappa\otimes\tilde f^{\otimes m^{(l)}}\in E^{(l)}_{[\tilde f]}$ for the point defined as in \cref{sec:sprays} with the weights $m^{(l)}$.
\end{setting}

\begin{lemma}\label[lemma]{lem:linearization}
Let $X$ be a Stein space, $K\subset X$ a compact set, $\Lambda$ a compact Hausdorff space and $\varphi\colon\Lambda\times X\to Y\setminus Z$ a continuous map such that $\varphi(\lambda,\cdot)$ is holomorphic for every $\lambda\in\Lambda$. Then there is $\delta>0$ with the following property. Let $P'$ be a compact Hausdorff space, $\lambda\colon P'\to\Lambda$ a continuous map and $g\colon P'\times\Op K\to Y\setminus Z$ a continuous map such that $g(p,\cdot)$ is holomorphic and $\sup_K\dist(g(p,\cdot),\varphi(\lambda(p),\cdot))<\delta$ for every $p\in P'$. Then there is a continuous map $\xi\colon P'\times\Op K\to E$ with $\xi(p,\cdot)\in\GO(\Op K,\varphi(\lambda(p),\cdot)^*E)$ and $t\xi(p,x)\in\Omega$ for $t\in[0,1]$, such that
\begin{enumerate}[(a)]
\item $s\circ\xi(p,\cdot)=g(p,\cdot)$ on $\Op K$,
\item $\xi(p,x)=0$ whenever $g(p,x)=\varphi(\lambda(p),x)$,
\item if $X_0\subset X$ is a closed analytic subvariety with ideal sheaf $\cI_{X_0}$ and $g(p,\cdot)$ and $\varphi(\lambda(p),\cdot)$ agree to order $\ell$ along $X_0$, then $\xi(p,\cdot)\in\Gamma(\Op K,\cI_{X_0}^{\ell}\cdot\cO(\varphi(\lambda(p),\cdot)^*E))$.
\end{enumerate}
\end{lemma}

\begin{proof}
Write $\beta_y=ds_{0_y}|_{E_y}\colon E_y\to T_y(Y\setminus Z)$. The graph of each $\varphi(\lambda,\cdot)$ is Stein and has a Stein neighborhood \cite[Theorem 3.1.1]{F17}. On it the surjective bundle map $\beta$ has a holomorphic right inverse \cite[Corollary 2.6.5]{F17}. After shrinking $\Op K$ to have compact closure, compactness of $\Lambda$ gives finitely many such inverses valid for nearby parameters. A partition of unity in $\lambda$ combines them into a continuous right inverse
\[
\mathcal R\colon\varphi^*T(Y\setminus Z)\longrightarrow\varphi^*E,
\]
holomorphic in $x$, with $\beta\circ\mathcal R=\id$.

Set $\Psi=s\circ\mathcal R$. Then $\Psi(\lambda,x,0)=\varphi(\lambda,x)$ and $d_v\Psi(\lambda,x,0)=\id$. The inverse function theorem with continuous parameters \cite[proof of Proposition 6.5.1]{F17}, uniformly on $\Lambda\times\overline{\Op K}$, gives a continuous inverse $\Theta(\lambda,x,y)$, holomorphic in $(x,y)$, defined whenever $\dist(y,\varphi(\lambda,x))<\delta$ for some $\delta>0$. It satisfies
\[
\Psi(\lambda,x,\Theta(\lambda,x,y))=y,\qquad
\Theta(\lambda,x,\varphi(\lambda,x))=0.
\]
We choose its values in a small fiber ball whose image under $\mathcal R$ lies in $\Omega$.

For $P'$, $\lambda$ and $g$ as in the statement, the distance bound extends to $P'\times\Op K$ after shrinking. Set
\[
\xi(p,x)=\mathcal R\bigl(\lambda(p),x,\Theta(\lambda(p),x,g(p,x))\bigr).
\]
Then $s\circ\xi(p,\cdot)=g(p,\cdot)$ and $t\xi(p,\cdot)\in\Omega$ by the fiber-ball choice. The identity for $\Theta$ gives (b). If $g(p,\cdot)$ and $\varphi(\lambda(p),\cdot)$ agree to order $\ell$ along $X_0$, then so do $x\mapsto\Theta(\lambda(p),x,g(p,x))$ and $x\mapsto\Theta(\lambda(p),x,\varphi(\lambda(p),x))=0$, which gives (c).
\end{proof}

\begin{lemma}\label[lemma]{lem:lifting}
In \cref{set:spray}, let $X$ be a Stein space, $K\subset X$ a compact set, $P_0\subset P$ compact Hausdorff spaces, $X_0\subset X$ a closed analytic subvariety, $L\to P\times X$ a continuous family of holomorphic line bundles and $\tilde f_0\in\GP(P\times X,L)^{\oplus(N+1)}$ such that $\tilde f_0(p,x)\ne0$ and $[\tilde f_0(p,x)]\in Y$ for all $(p,x)\in P\times X$. Let $f_t\colon P\times X\to Y\setminus Z$ \textup{(}$t\in[0,1]$\textup{)} be a homotopy of continuous families of holomorphic maps such that $f_0=[\tilde f_0]$, $f_t(p,\cdot)=f_0(p,\cdot)$ for $p\in P_0$, and $f_t(p,\cdot)$ agrees with $f_0(p,\cdot)$ to order $\ell$ along $X_0$ for all $t$ and $p$. Then there are $l\in\N$ and a homotopy $h_t\in\GP(P\times\Op K,E^{(l)}(L))$ \textup{(}$t\in[0,1]$\textup{)} such that
\begin{enumerate}[(i)]
\item $h_0=0$, $h_t(p,\cdot)=0$ for $p\in P_0$, and $h_t(p,\cdot)\in\Gamma(\Op K,\cI_{X_0}^{\ell}\cdot\cO(E^{(l)}(L)_p))$, where $\cI_{X_0}$ is the ideal sheaf of $X_0$,
\item $h_t(p,x)\otimes\tilde f_0(p,x)^{\otimes m^{(l)}}\in\Omega^{(l)}$ and $s^{(l)}\bigl(h_t(p,x)\otimes\tilde f_0(p,x)^{\otimes m^{(l)}}\bigr)=f_t(p,x)$ for $(p,x)\in P\times\Op K$.
\end{enumerate}
\end{lemma}

\begin{proof}
Following \cite[Proposition 6.5.1]{F17} and \cite[Lemma 2.6]{BK23}, apply \cref{lem:linearization} with $\Lambda=[0,1]\times P$ and $\varphi(t,p,\cdot)=f_t(p,\cdot)$. Compactness gives a partition $0=t_0<\cdots<t_l=1$ such that, uniformly in $p$, $f_t(p,\cdot)$ is so close to $f_{t_j}(p,\cdot)$ on $K$ for $t_j\le t\le t_{j+1}$ that \cref{lem:linearization} applies.

Starting with $h^{(0)}=0$ and $R^{(0)}=\id$, suppose that a lift $h^{(j)}$ on $[0,t_j]$ has been constructed. Its endpoint gives the tuple without zeros
\[
\tilde f_{t_j}(p,\cdot)=R^{(j)}\bigl(\tilde f_0(p,\cdot);h^{(j)}(t_j,p,\cdot)\bigr)
\in\GP(P\times\Op K,L^{d^j})^{\oplus(N+1)}.
\]
Linearization on $[t_j,t_{j+1}]$ gives $\xi_t(p,\cdot)\in\GO(\Op K,f_{t_j}(p,\cdot)^*E)$ with $s\circ\xi_t(p,\cdot)=f_t(p,\cdot)$ and values in $\Omega$. It vanishes at $t_j$, on $P_0$ and to order $\ell$ along $X_0$. By \cref{prop:substitution}(iii),
\[
\xi_t(p,\cdot)=\kappa_{j+1}(t,p,\cdot)\otimes\tilde f_{t_j}(p,\cdot)^{\otimes m},
\]
where $\kappa_{j+1}$ is a section of the block $\bigoplus_\mu(L^{-d^jm_\mu})^{\oplus r_\mu}$ with the same vanishing properties. Define
\[
h^{(j+1)}(t,p,\cdot)=\bigl(h^{(j)}(\min\{t,t_j\},p,\cdot),\,
\kappa_{j+1}(\max\{t,t_j\},p,\cdot)\bigr).
\]
This is continuous and satisfies (i). For $t\le t_j$, its last block is zero and $R_0(w)\in\C^*w$ over $Y\setminus Z$, so appending that block leaves the projectivized image unchanged. For $t\ge t_j$, \eqref{eq:iterate} identifies the last step with $s\circ\xi_t(p,\cdot)=f_t(p,\cdot)$, and all intermediate values remain in $Y\setminus Z$. Hence (ii) holds. Then $h_t:=h^{(l)}(t,\cdot,\cdot)$ has the required properties.
\end{proof}

\section{Proof of the main theorems}\label{sec:proof}

\subsection{The version with conic lifts}
\begin{theorem}\label[theorem]{thm:lift}
Let $X$, $G$, $K$, $X_0$, $\ell$ and $P_0\subset P$ be as in \cref{thm:main2}. Let $Y\subset\PP^N$ be a smooth projective variety and $s\colon\cO_Y(m;r)\merto Y$ a rational spray with a conic lift $R\colon C(Y)\times\C^{|r|}\to C(Y)$ of degree $d$. Assume $Z(R)\ne Y$ and put $Z:=Z(R)$. Let $f_1\colon P\times\Op K\to Y\setminus Z$ be a continuous family of holomorphic maps satisfying $(\mathrm H_Z)$, and assume moreover
\begin{enumerate}
\item[(C)] $R_0(\tilde f_0(p,\cdot))\not\equiv0$ on every irreducible component of $X$ for every $p\in P$.
\end{enumerate}
Then there are $L$, $\sigma$ as in \cref{thm:main2}\textup{(2)} and $l\in\N$ such that conclusions \textup{(i)--(vii)} hold with the following replacements. In \textup{(v)},
\[
L(F)=L(f_0)^{d^l}\otimes\pr_X^*L^k.
\]
In \textup{(vi)} and \textup{(vii)}, replace $\tilde f_0$ by the section $R^{(l)}_0(\tilde f_0)$ of $L(f_0)^{d^l}\otimes\C^{N+1}$.
\end{theorem}

On the irreducible components of $X$ which meet $K$, condition \textup{(C)} follows from $(\mathrm H_Z)$, since $f_0$ maps $P\times\Op K$ into $Y\setminus Z(R)$. Condition \textup{(C)} ensures that the degree-zero part $R^{(l)}_0(\tilde f_0)$ of the iterated tuple vanishes identically on no irreducible component of $X$. It is automatic for the normalized lifts $R_0=\id$ of \cref{prop:normalized}. We retain arbitrary degree in this theorem to make the dependence on a prescribed conic lift explicit.

\begin{proof}[Proof of \cref{thm:lift}]
We are in \cref{set:spray}. Write $R^{(l)}=\sum_{j=0}^{\nu_l}R^{(l)}_j$ for the iterated lifts, $\nu_l:=\deg_\zeta R^{(l)}$. By $(\mathrm H_Z)$, $\tilde f_0$ has no zeros and $f_t$ takes values in $Y\setminus Z$ on $P\times\Op K$.

\emph{Step 1: the section $\sigma$.} Let $L\in G$ and $\sigma\in\GO(X,L)$ be obtained from \cref{lem:hull} applied to the neighborhood $\Op K$ in $(\mathrm H_Z)$. Thus $\{\sigma=0\}\cap K=\emptyset$, $\sigma$ vanishes identically on no irreducible component of $X$ or of $X_0$, and $\widehat K_{X\setminus\{\sigma=0\}}\subset\Op K$. Let $X''$ be the union of the irreducible components of $X$ which do not meet $\widehat K_{X\setminus\{\sigma=0\}}$. The irreducible components form a locally finite family, so $X''$ is a closed analytic subvariety and its ideal sheaf is coherent. Put $\cI:=\cI_{X_0}^{\ell}\cdot\cI_{X''}$, where $\cI_{X_0}$ and $\cI_{X''}$ are the ideal sheaves of $X_0$ and $X''$. Shrinking $\Op \widehat K_{X\setminus\{\sigma=0\}}$, we always assume that it does not meet $\{\sigma=0\}\cup X''$. There $\cI=\cI_{X_0}^{\ell}$. With $X$ and $X_0$ fixed, $L$, $\sigma$ and $X''$ are chosen from $K$, $G$ and the fixed neighborhood in $(\mathrm H_Z)$, independently of the precision required in \textup{(iv)}.

\emph{Step 2: lifting.} Apply \cref{lem:lifting} on a Stein neighborhood of $\widehat K_{X\setminus\{\sigma=0\}}$. It gives $l\in\N$ and $h_t\in\GP(P\times\Op\widehat K_{X\setminus\{\sigma=0\}},E^{(l)}(L(f_0)))$ such that
\[
s^{(l)}\bigl(h_t\otimes\tilde f_0^{\otimes m^{(l)}}\bigr)=f_t,
\qquad h_t\otimes\tilde f_0^{\otimes m^{(l)}}\in\Omega^{(l)}.
\]
Moreover $h_0=0$, $h_t=0$ on $P_0$, and $h_t$ belongs to $\cI_{X_0}^{\ell}$.

\emph{Step 3: approximation.} Put $\Phi(p,x,\zeta)=\zeta\otimes\tilde f_0(p,x)^{\otimes m^{(l)}}$. Choose an open tube $\mathcal U$ about the graph of $h_1$ such that $\Phi(\mathcal U)\subset\Omega^{(l)}$ and $s^{(l)}\circ\Phi$ is as close to $f_1$ on $\mathcal U$ as required in \textup{(iv)}. By \cref{prop:runge-interp}, there are $n\in\Z_{\ge0}$ and $\eta\in\GP(P\times X,E^{(l)}(L(f_0))\otimes\pr_X^*L^n)$ with components in $\cI$, vanishing on $P_0$, such that $\eta/\sigma^n$ approximates $h_1$ uniformly on $P\times\widehat K_{X\setminus\{\sigma=0\}}$. Choose the approximation so close that the segment from $\eta/\sigma^n$ to $h_1$ lies in $\mathcal U$ over $P\times\widehat K_{X\setminus\{\sigma=0\}}$, hence over $P\times\Op\widehat K_{X\setminus\{\sigma=0\}}$ after shrinking. This is possible by compactness. Since $\cI\subset\cI_{X''}$, $\eta=0$ on $X''$.

\emph{Step 4: substitution.} Put $k:=\nu_ln$ and
\[
\widetilde F(p,\cdot):=(R^{(l)})^h\bigl(\tilde f_0(p,\cdot),\eta(p,\cdot),\sigma^n\bigr),
\]
where $(R^{(l)})^h$ is the homogenization of $R^{(l)}$ (\cref{rem:spray-basic}(2)). By \cref{prop:substitution} (for the conic lift $R^{(l)}$ of degree $d^l$, \cref{lem:iteration}), $\widetilde F\in\GP(P\times X,L(f_0)^{d^l}\otimes\pr_X^*L^k)^{\oplus(N+1)}$, and $[\widetilde F(p,x)]\in Y$ whenever $\widetilde F(p,x)\ne0$. By Step 3, $\Phi(p,x,\eta(p,x)/\sigma(x)^n)\in\Omega^{(l)}$ for $x\in\Op \widehat K_{X\setminus\{\sigma=0\}}$, so $\widetilde F(p,x)\ne0$ for all $(p,x)\in P\times\Op \widehat K_{X\setminus\{\sigma=0\}}$, and with $F(p,\cdot):=[\widetilde F(p,\cdot)]$ we have
\[
F(p,x)=s^{(l)}\Bigl(\bigl(\eta(p,x)/\sigma(x)^n\bigr)\otimes\tilde f_0(p,x)^{\otimes m^{(l)}}\Bigr)\qquad(x\in\Op \widehat K_{X\setminus\{\sigma=0\}}).
\]

We check that $\widetilde F(p,\cdot)$ does not vanish identically on any irreducible component $X'$ of $X$. If $X'\cap \widehat K_{X\setminus\{\sigma=0\}}\ne\emptyset$, then $\widetilde F(p,\cdot)$ has no zeros on $X'\cap\Op \widehat K_{X\setminus\{\sigma=0\}}\ne\emptyset$. Otherwise $X'\subset X''$, so $\eta(p,\cdot)=0$ on $X'$ and $\widetilde F(p,\cdot)=\sigma^kR^{(l)}_0(\tilde f_0(p,\cdot))$ on $X'$. Condition (C) gives $x\in X'$ with $R_0(\tilde f_0(p,x))\ne0$. Since $R_0(w)=\gamma(w)w$ (\cref{rem:spray-basic}(1)), this means $\tilde f_0(p,x)\ne0$ and $\gamma(\tilde f_0(p,x))\ne0$, and by the homogeneity of $R_0$ inductively $R^{(l)}_0(\tilde f_0(p,x))\in\C^*\tilde f_0(p,x)$, so $R^{(l)}_0(\tilde f_0(p,\cdot))\not\equiv0$ on $X'$, and $\sigma\not\equiv0$ on $X'$. Hence $F=[\widetilde F]$ is a continuous family of meromorphic maps with $L(F)=L(f_0)^{d^l}\otimes\pr_X^*L^k$, holomorphic on $P\times\Op \widehat K_{X\setminus\{\sigma=0\}}$. The same argument shows that $R^{(l)}_0(\tilde f_0(p,\cdot))$ does not vanish identically on any irreducible component. Since $R^{(l)}_0(\tilde f_0(p,\cdot))\in\C^*\tilde f_0(p,\cdot)$ wherever it does not vanish, we obtain $[R^{(l)}_0(\tilde f_0(p,\cdot))]=f_0(p,\cdot)$.

\emph{Verification of \textup{(v)--(vii)}.} (v) holds by construction. (vi): for $p\in P_0$, $\eta(p,\cdot)=0$ and $\widetilde F(p,\cdot)=\sigma^kR^{(l)}_0(\tilde f_0(p,\cdot))$. (vii): the components of $\widetilde F(p,\cdot)-\sigma^kR^{(l)}_0(\tilde f_0(p,\cdot))$ are sums of terms of positive degree in $\eta(p,\cdot)$, whose components lie in $\cI\subset\cI^{\ell}_{X_0}$.

\emph{Step 5: the homotopy \textup{(i)--(iv)}.} For $t\in[0,1]$, $p\in P$ and $x\in\Op \widehat K_{X\setminus\{\sigma=0\}}$ put
\[
\hat f_t(p,x):=s^{(l)}\Bigl(\bigl((1-t)\,\eta(p,x)/\sigma(x)^n+t\,h_1(p,x)\bigr)\otimes\tilde f_0(p,x)^{\otimes m^{(l)}}\Bigr).
\]
This is well defined by Step 3, holomorphic in $x$ and continuous in $(t,p,x)$. By Step 4, $\hat f_0=F$, and $\hat f_1=f_1$ by Step 2. By the choice of $\mathcal U$, \textup{(iv)} holds. If $p\in P_0$, then $\eta(p,\cdot)=0=h_1(p,\cdot)$ and $\hat f_t(p,\cdot)=f_0(p,\cdot)=f_1(p,\cdot)$. Since $\sigma$ is invertible on $\Op \widehat K_{X\setminus\{\sigma=0\}}$, the components of $(1-t)\,\eta(p,\cdot)/\sigma^n+t\,h_1(p,\cdot)$ lie in $\cI^{\ell}_{X_0}$ there. Hence $\hat f_t(p,\cdot)$ agrees with $s^{(l)}(0\otimes\tilde f_0(p,\cdot)^{\otimes m^{(l)}})=f_0(p,\cdot)$ to order $\ell$ along $X_0\cap\Op \widehat K_{X\setminus\{\sigma=0\}}$, and $f_0(p,\cdot)$ agrees with $f_1(p,\cdot)$ by $(\mathrm H_Z)$ with $t=1$. This proves (iii).
\end{proof}

\subsection{Proofs of \texorpdfstring{\cref{thm:general,thm:main2}}{Theorems 1.7 and 1.2}}\label{sec:proof-main2}
\begin{proof}[Proof of \cref{thm:general}]
By \cref{prop:normalized}(a),(b), there are a rational spray $s'\circ \varpi$ on $\cO_Y(-m)^{\oplus r'}$ and a conic lift $R$ of it of degree $1$ with $R_0(w)=w$ and $Z(R)\subset Z(s)\subset Z$. Hence $Y\setminus Z\subset Y\setminus Z(R)$. Condition (C) of \cref{thm:lift} holds: since $R_0(w)=w$, we have $R_0(\tilde f_0(p,\cdot))=\tilde f_0(p,\cdot)$, which vanishes identically on no irreducible component of $X$ by \cref{def:families}. Moreover $R^{(l)}_0=R_0^{\circ l}=\id$ (\cref{lem:iteration}(a)), so $R^{(l)}_0(\tilde f_0)=\tilde f_0$ in the conclusions of \cref{thm:lift}.

(2) Under $(\mathrm H_Z)$, \cref{thm:lift} applies to $R$ (note $Y\setminus Z\subset Y\setminus Z(R)$) and gives the conclusions (i)--(vii) of \cref{thm:main2}(2).

(1) By (i)--(iv) of \cref{thm:main2}(2) with $t=0$, $(\mathrm H_Z)$ implies (A).

(A)$\Rightarrow(\mathrm H_Z)$: Shrinking $\Op K$, we may assume that it is Stein by \cref{lem:hull}. Use \cref{set:spray} for the normalized spray $s'\circ \varpi$, with the prescribed closed set $Z\supset Z(R)$. Thus $\Omega$ consists only of points whose base and image lie in $Y\setminus Z$. Since $f_1(P\times K)\subset Y\setminus Z$ is compact, by (A) we may choose $F=[\widetilde F]$ so close to $f_1$ uniformly on $P\times K$ that $F$ takes values in $Y\setminus Z$ near $P\times K$ and \cref{lem:linearization} applies with $\Lambda:=P$, $\varphi:=f_1$ and $g:=F$. Moreover $\widetilde F$ has no zeros near $P\times K$. Linearization gives $\xi(p,\cdot)\in\GO(\Op K,f_1(p,\cdot)^*\cO_Y(-m)^{\oplus r'})$, with $(s'\circ \varpi)\circ\xi(p,\cdot)=F(p,\cdot)$, $t\xi(p,\cdot)\in\Omega$, $\xi(p,\cdot)=0$ for $p\in P_0$, and $\xi(p,\cdot)$ vanishing to order $\ell$ along $X_0$. Set $f_0:=F$, $\tilde f_0:=\widetilde F$ and $f_t(p,\cdot):=(s'\circ \varpi)\circ((1-t)\xi(p,\cdot))$. Then $f_t(p,\cdot)$ and $f_0(p,\cdot)$ agree with $f_1(p,\cdot)$ to order $\ell$ along $X_0$, which gives the interpolation conditions, and the whole homotopy takes values in $Y\setminus Z$ by the choice of $\Omega$. This proves $(\mathrm H_Z)$.
\end{proof}

\begin{proof}[Proof of \cref{thm:main2}]
Apply \cref{thm:general} with $Z=\emptyset$ to an algebraic dominating spray. By \cref{prop:normalized}(c), its degree-one conic lift $R$ has $R_0=\id$ and no zeros on $\widehat Y\times\C^{r}$, and the same holds for $R^{(l)}$.

For the assertion following the theorem, on $X\setminus\{\sigma=0\}$ we have $\widetilde F=\sigma^kR^{(l)}(\tilde f_0,\eta/\sigma^n)$ by \cref{prop:substitution}(ii), so $\widetilde F(p,\cdot)$ has no zeros there wherever $\tilde f_0(p,\cdot)$ has none. If $f_0(p,\cdot)$ is holomorphic near a point of $X\setminus\{\sigma=0\}$, write locally $\tilde f_0(p,\cdot)=a u$ with $u$ a holomorphic map to $\C^{N+1}$ without zeros. Equivariance gives
\[
R^{(l)}\bigl(a u,(\eta/\sigma^n)(p,\cdot)\bigr)=a\,R^{(l)}\bigl(u,a^{m^{(l)}}\cdot(\eta/\sigma^n)(p,\cdot)\bigr).
\]
After canceling $a$, the right side is holomorphic and nonzero, since all weights are nonnegative. It therefore extends $F(p,\cdot)$ holomorphically across the removable common zeros. If $K$ is $\cO(X)$-convex, then $\sigma=1$ by \cref{lem:hull}.
\end{proof}

\subsection{Proof of \texorpdfstring{\cref{thm:main1}}{Theorem 1.1}}\label{sec:proof-main1}

\begin{proof}[Proof of \cref{thm:main1}]
Let $s\colon\cO_Y(-m)^{\oplus r}\to Y$ be an algebraic dominating spray with a conic lift $R$ of degree $1$ such that $R_0(w)=w$ and $R$ has no zeros on $\widehat Y\times\C^{r}$ (\cref{prop:normalized}(c)). Then $Z=\emptyset$, $E^{(l)}=\cO_Y(-m)^{\oplus rl}$, $\Omega$ and $\Omega^{(l)}$ are the whole bundles, and $s^{(l)}$ is a morphism. Let $\Op K$ be as in the hypothesis of \cref{thm:main1}, and let $L$ and $\sigma$ be obtained from \cref{lem:hull} applied to $\Op K$.

\emph{Step 1.} Since $Y$ is Oka \cite[Corollary 5.6.14]{F17}, POPAI \cite[Theorem 1.2]{F23} gives a homotopy $f^t\colon P\times X\to Y$ \textup{(}$t\in[0,1]$\textup{)} with $f^0=f_0$, $f^1$ holomorphic in $x$ on $X$ and $f^t=f_0$ on $(P_0\times X)\cup(P\times X_0)$. Put $g:=f^1$, $L(g):=g^*\cO_Y(1)$ and $\tilde g:=g^*w$. Then $\tilde g$ has no zeros.

\emph{Step 2.} The homotopy $f^t$ is only continuous. We now replace it near $\widehat K_{X\setminus\{\sigma=0\}}$ by a homotopy of holomorphic maps. Shrinking $\Op\widehat K_{X\setminus\{\sigma=0\}}$, we may assume that it is a Stein neighborhood of $\widehat K_{X\setminus\{\sigma=0\}}$ contained in $\Op K\setminus\{\sigma=0\}$. This is possible since $\widehat K_{X\setminus\{\sigma=0\}}\subset\Op K$ by the choice of $\sigma$, and since $\widehat K_{X\setminus\{\sigma=0\}}$ has a basis of Stein neighborhoods in the Stein space $X\setminus\{\sigma=0\}$. By POPAI \cite[Theorem 1.2]{F23} (see also \cite[Corollary 5.7]{Ku21}) applied on $\Op\widehat K_{X\setminus\{\sigma=0\}}$ to $(t,p)\mapsto f^t(p,\cdot)$, with $Q_0:=(\{0,1\}\times P)\cup([0,1]\times P_0)\subset Q:=[0,1]\times P$ and the subvariety $X_0\cap\Op\widehat K_{X\setminus\{\sigma=0\}}$, there is a homotopy $f^{(t,\tau)}\colon P\times\Op\widehat K_{X\setminus\{\sigma=0\}}\to Y$ \textup{(}$\tau\in[0,1]$\textup{)} with $f^{(t,0)}=f^t$, $f^{(t,1)}(p,\cdot)$ holomorphic on $\Op\widehat K_{X\setminus\{\sigma=0\}}$, and $f^{(t,\tau)}(p,x)=f^t(p,x)$ for $(t,p,x)\in(Q_0\times\Op\widehat K_{X\setminus\{\sigma=0\}})\cup(Q\times(X_0\cap\Op\widehat K_{X\setminus\{\sigma=0\}}))$. Let $\chi\colon X\to[0,1]$ be a continuous function with compact support in $\Op\widehat K_{X\setminus\{\sigma=0\}}$ which equals $1$ near $\widehat K_{X\setminus\{\sigma=0\}}$, and put $f_\chi^t(p,x):=f^{(t,\chi(x))}(p,x)$ for $x\in\Op\widehat K_{X\setminus\{\sigma=0\}}$ and $f_\chi^t(p,x):=f^t(p,x)$ otherwise. Then $f_\chi^t$ is a homotopy from $f_0$ to $g$ on $P\times X$, equal to $f_0$ on $(P_0\times X)\cup(P\times X_0)$, and $f_\chi^t(p,\cdot)$ is holomorphic near $\widehat K_{X\setminus\{\sigma=0\}}$.

\emph{Step 3.} Shrinking $\Op\widehat K_{X\setminus\{\sigma=0\}}$, we may assume that it is Stein and that $f_\chi^t(p,\cdot)$ is holomorphic on it. Applying \cref{lem:lifting} to the homotopy $t\mapsto f_\chi^{1-t}$ from $g$ to $f_0$ on $\Op\widehat K_{X\setminus\{\sigma=0\}}$ and shrinking $\Op\widehat K_{X\setminus\{\sigma=0\}}$ again, we obtain $l$ and a homotopy $h_t\in\GP(P\times\Op\widehat K_{X\setminus\{\sigma=0\}},E^{(l)}(L(g)))$ with $f_\chi^{1-t}(p,\cdot)=s^{(l)}(h_t(p,\cdot)\otimes\tilde g(p,\cdot)^{\otimes m^{(l)}})$, $h_0=0$, and $h_t(p,\cdot)=0$ for $p\in P_0$ and on $X_0\cap\Op\widehat K_{X\setminus\{\sigma=0\}}$. Choose $\eta$ as in Step 3 of the proof of \cref{thm:lift} (with $\cI:=\cI_{X_0}$, the ideal sheaf of $X_0$). Put $\nu_l:=\deg_\zeta R^{(l)}$, $k:=\nu_l n$, and define $\widetilde F:=(R^{(l)})^h(\tilde g,\eta,\sigma^n)$. Since $\tilde g$ has no zeros and $R^{(l)}$ has no zeros on $\widehat Y\times\C^{rl}$, \cref{prop:substitution}(ii) shows that $F=[\widetilde F]$ is holomorphic on $P\times(X\setminus\{\sigma=0\})$ with $F(p,\cdot)=s^{(l)}((\eta(p,\cdot)/\sigma^n)\otimes\tilde g(p,\cdot)^{\otimes m^{(l)}})$ there, and $L(F)=L(g)\otimes\pr_X^*L^k$.

\emph{Step 4.} For $(t,\vartheta)\in[0,1]^2$ and $x\in\Op\widehat K_{X\setminus\{\sigma=0\}}$ put
\[
\Lambda_{t,\vartheta}(p,x):=s^{(l)}\Bigl(\bigl((1-\vartheta)(\eta/\sigma^n)(p,x)+\vartheta h_t(p,x)\bigr)\otimes\tilde g(p,x)^{\otimes m^{(l)}}\Bigr).
\]
Then $\Lambda_{t,0}=F$, $\Lambda_{t,1}=f_\chi^{1-t}$, and $\Lambda_{1,\vartheta}=\hat f_\vartheta$ is the homotopy of Step 5 of the proof of \cref{thm:lift}. Since $h_0=0$, we have $\Lambda_{0,\vartheta}=s^{(l)}((1-\vartheta)(\eta/\sigma^n)\otimes\tilde g^{\otimes m^{(l)}})$. We use this formula to define $\Lambda_{0,\vartheta}(p,x)$ for all $x\in X\setminus\{\sigma=0\}$. Let $\chi'\colon X\to[0,1]$ be a continuous function such that $1-\chi'$ has compact support in $\Op\widehat K_{X\setminus\{\sigma=0\}}$ and $\chi'=0$ near $\widehat K_{X\setminus\{\sigma=0\}}$. For $t\in[0,1]$ and $(p,x)\in P\times(X\setminus\{\sigma=0\})$ define
\[
f_t(p,x):=\begin{cases}f_\chi^{\,2t\,\chi'(x)}(p,x),&t\le1/2,\\ \Lambda_{1-\chi'(x),2-2t}(p,x),&t\ge1/2,\ x\in\Op\widehat K_{X\setminus\{\sigma=0\}},\\ \Lambda_{0,2-2t}(p,x),&t\ge1/2,\ x\notin\operatorname{supp}(1-\chi').\end{cases}
\]
The expressions agree where they overlap: at $t=1/2$ because $\Lambda_{1-\chi',1}=f_\chi^{\,\chi'}$, and on $\Op\widehat K_{X\setminus\{\sigma=0\}}\setminus\operatorname{supp}(1-\chi')$ because $\chi'=1$ and $h_0=0$ there. Hence $f_t$ is a homotopy of continuous maps from $f_0$ to $F$ on $P\times(X\setminus\{\sigma=0\})$.
Near $\widehat K_{X\setminus\{\sigma=0\}}$ we have $\chi'=0$, so $f_t=f_0$ for $t\le1/2$ and $f_t=\hat f_{2-2t}$ afterwards. This gives (i). On $P_0$ and $X_0\setminus\{\sigma=0\}$, the homotopy $f_\chi$ is fixed and $h=\eta=0$, giving (ii). For $p\in P_0$, $\widetilde F(p,\cdot)=\sigma^k\tilde g(p,\cdot)$, hence $F(p,\cdot)=g(p,\cdot)=f_0(p,\cdot)$ as meromorphic maps. Density in each component of $X$ and $X_0$ follows from \cref{lem:hull}.

\emph{Jets.} In Step 1 apply the relative jet version of the Oka principle \cite[Theorem 1.4]{FP01} to $X\times Y\to X$, with empty approximation set. Thus every $f^t(p,\cdot)$ is holomorphic on a common neighborhood of $X_0$ and agrees with $f_0(p,\cdot)$ to order $\ell$ along $X_0$. Apply the same theorem in Step 2 with parameter pair $(Q,Q_0)$ on $\Op\widehat K_{X\setminus\{\sigma=0\}}$. Since $(Q,Q_0)$ is again a finite CW pair, $Q_0$ is a strong deformation retract of a neighborhood in $Q$, as required in \cite[Definition 3]{FP01}. The resulting homotopy $f^{(t,\tau)}$ preserves the jets of $f^t(p,\cdot)$ along $X_0\cap\Op\widehat K_{X\setminus\{\sigma=0\}}$. The cut-off $\chi$ used to form $f_\chi$ equals $1$ near $\widehat K_{X\setminus\{\sigma=0\}}$, so the homotopy to be lifted has those jets there. In Step 3 use $\cI:=\cI_{X_0}^{\ell}$. The components of $\widetilde F(p,\cdot)-\sigma^k\tilde g(p,\cdot)$ then lie in $\cI_{X_0}^{\ell}$ globally, because all other terms have positive degree in $\eta(p,\cdot)$. On $X\setminus\{\sigma=0\}$ the tuples $\widetilde F(p,\cdot)$ and $\sigma^k\tilde g(p,\cdot)$ have no zeros, so $F(p,\cdot)$ and $g(p,\cdot)$ agree to order $\ell$ along $X_0\setminus\{\sigma=0\}$. Since $g(p,\cdot)$ agrees with $f_0(p,\cdot)$ to order $\ell$ along $X_0$, this proves the asserted interpolation.
\end{proof}

\begin{remark}\label[remark]{rem:torus}
\Cref{thm:main1} does not hold for analytically elliptic manifolds in general. Indeed, the holomorphic function $f(z)=1/z$ on a neighborhood of an annulus $A\subset\C$ clearly extends to a continuous function $\tilde f\colon\C\to\C$. Consider an elliptic curve $Y=\C/\Lambda$ with quotient map $\pi\colon\C\to Y$, which is elliptic but not algebraically elliptic, since it is not unirational (\cref{thm:ME}), and note that every compact subset of $\C$ is $G$-convex for $G=\{[\cO_\C]\}$. If $\pi\circ f$ were approximated on $A$ by a meromorphic map $\C\merto Y$, this map would be holomorphic and of the form $\pi\circ g$ with $g\in\cO(\C)$, so $g-\lambda$ would approximate $f$ on $A$ for some $\lambda\in\Lambda$. This contradicts $\int_{|z|=r}f\,dz=2\pi i$ for a circle $\{|z|=r\}\subset A$. Moreover, $t\mapsto\pi\circ(tf)$ is a homotopy of holomorphic maps on $\Op A$ from a constant map to $\pi\circ f$, so the implication (H)$\Rightarrow$(A) of \cref{thm:main2} fails for $Y$ as well. For projective manifolds, analytic ellipticity is equivalent to the Oka property \cite[Theorem 1.1]{FL26}. An algebraic analog is \cite[Theorem 1.3]{BK03}: if every holomorphic map from a disc in a nonsingular algebraic curve $C$ to a nonsingular projective variety $Y$ can be approximated by regular maps $C\to Y$, then the first Betti number of $Y$ vanishes.
\end{remark}

\section{Rational sprays and birational geometry}\label{sec:birational}

\subsection{Algebraic sprays from rational sprays}
For an affine variety $X$ we write $\cO^{\mathrm{alg}}(X)$ for the ring of regular functions on $X$, and $\cI^{\mathrm{alg}}(Z)\subset\cO^{\mathrm{alg}}(X)$ for the ideal of a closed subset $Z\subset X$. The superscript distinguishes them from the sheaves of holomorphic functions. For a homogeneous polynomial $h$ on $\C^{N+1}$ and $Y\subset\PP^N$ we put $V(h):=Y\cap\{h=0\}$ and $D(h):=Y\setminus V(h)$.

\begin{lemma}\label[lemma]{lem:localization}
Let $Y\subset\PP^N$ be a smooth projective variety and $h$ a homogeneous polynomial of degree $d$ which does not vanish identically on $C(Y)$, regarded as a section of $\cO_Y(d)$, and let $s_h\colon D(h)\times\C^{r}\to Y$ be a morphism with $s_h(y,0)=y$. For $k\gg0$, the map $s\colon\cO_Y(-dk)^{\oplus r}\to Y$ defined by
\[
s(v_y):=\begin{cases}s_h\bigl(y,\langle h^k(y),v_y\rangle\bigr)&\text{if }y\in D(h),\\ y&\text{if }y\in V(h)\end{cases}\qquad\bigl(v_y\in\cO_Y(-dk)_y^{\oplus r}\bigr)
\]
is a morphism. Here $h^k$ is a section of $\cO_Y(dk)$, and the brackets denote the componentwise duality pairing between $\cO_Y(dk)_y$ and $\cO_Y(-dk)_y^{\oplus r}$.
\end{lemma}

\begin{proof}
On $D(h)$, $v_y\mapsto(y,\langle h^k(y),v_y\rangle)$ is an isomorphism of vector bundles onto $D(h)\times\C^{r}$, so $s$ is regular over $D(h)$. It remains to prove regularity near the points over $V(h)$. Fix $j$ and regard $Y_j:=Y\cap\{w_j\ne0\}$ as an affine variety in $\C^N$ via the affine coordinates $w_i/w_j$ ($i\ne j$), which we denote by $z_1,\dots,z_N$. We use $(z_1,\dots,z_{N+r})$ as coordinates on $Y_j\times\C^{r}\subset\C^{N+r}$. Let $e_j(y)\in\cO_Y(-1)_y$ be the point with $j$-th coordinate $1$ and write $v_y=(z_{N+1}e_j(y)^{\otimes dk},\dots,z_{N+r}e_j(y)^{\otimes dk})$. Then $\langle h^k(y),v_y\rangle=h_j(y)^k(z_{N+1},\dots,z_{N+r})$, where $h_j(w_0,\dots,w_N):=h(w_0,\dots,w_N)/w_j^d$. Let $V\subset Y_j\times\C^{r}$ be the complement of $s_h^{-1}(\{w_j\ne0\})$. Then $V\cap(Y_j\times\{0\})\subset(V(h)\cap Y_j)\times\{0\}$, since $s_h(y,0)=y$. For $1\le i\le N$ let $(s_h)_i:=z_i\circ s_h$. It is regular on $(Y_j\times\C^{r})\setminus V$. If $\cI^{\mathrm{alg}}(V)=(g_1,\dots,g_c)$, then $(s_h)_i$ is regular on $\{g_\mu\ne0\}$, so $g_\mu^M(s_h)_i$ is regular for some $M$ and all $\mu$. Since $\cI^{\mathrm{alg}}(V)^{c(M-1)+1}\subset(g_1^M,\dots,g_c^M)$, we obtain $(s_h)_i\cdot\cI^{\mathrm{alg}}(V)^n\subset\cO^{\mathrm{alg}}(Y_j\times\C^{r})$ for sufficiently large $n$. We fix such $n$. The ideal $\cI^{\mathrm{alg}}(V)^n+(z_{N+1},\dots,z_{N+r})\subset\cO^{\mathrm{alg}}(Y_j\times\C^{r})$ has zero set $V\cap(Y_j\times\{0\})$, which is contained in $(V(h)\cap Y_j)\times\{0\}$, where $h_j$ vanishes. By the Nullstellensatz, $h_j^M\in\cI^{\mathrm{alg}}(V)^n+(z_{N+1},\dots,z_{N+r})$ for some $M$, that is, there is $Q\in\cI^{\mathrm{alg}}(V)^n$ with $h_j^M-Q\in(z_{N+1},\dots,z_{N+r})$. In particular $Q(y,0)=h_j(y)^M$ for $y\in Y_j$. Then $P_i:=(s_h)_iQ\in\cO^{\mathrm{alg}}(Y_j\times\C^{r})$ and, since $(s_h)_i(y,0)=z_i(y)$, $P_i(y,0)=z_i(y)h_j(y)^M$, first on $D(h)\cap Y_j$ and then on all of $Y_j$ by density. Since $Q$ vanishes on $V$, we have $(s_h)_i(z)=P_i(z)/Q(z)$ for every $z\in Y_j\times\C^{r}$ with $Q(z)\ne0$, and after substituting $h_j^k(z_{N+1},\dots,z_{N+r})$, every nonconstant fiber term in $P_i$ and $Q$ is divisible by $h_j^k$. Dividing the resulting numerator and denominator by $h_j^M$ gives
\[
(s_h)_i\bigl(y,h_j(y)^k(z_{N+1},\dots,z_{N+r})\bigr)=\frac{z_i(y)+h_j(y)^{k-M}\widetilde P_i}{1+h_j(y)^{k-M}\widetilde Q}
\]
with $\widetilde P_i,\widetilde Q\in\cO^{\mathrm{alg}}(Y_j\times\C^{r})$. For $k>M$ the denominator equals $1$ on $\{h_j=0\}$, so the right-hand side is regular near every point over $V(h)\cap Y_j$ and equals $z_i(y)$ there. Hence $s$ maps a neighborhood of every such point into the affine variety $Y_j$, and its coordinate functions $z_i\circ s$ are regular there. Since a map into an affine variety is regular if and only if its coordinate functions are regular, $s$ is regular near such points. Finitely many charts cover $Y$.
\end{proof}

\begin{theorem}\label[theorem]{thm:alg-spray}
Let $Y\subset\PP^N$ be a smooth projective variety and $s\colon E\merto Y$ a rational spray of rank $r\ge1$, and let $\overline{I(s)}$ be the closure of $I(s)\subset E$ in $\PP(E\oplus\cO_Y)$. For $y\in Y$ put $I(s)_y:=I(s)\cap E_y$ and $\overline{I(s)}_y:=\overline{I(s)}\cap\PP(E_y\oplus\C)$. Let $Z$ be the union of $Z(s)$ and $\{y\in Y:\dim\overline{I(s)}_y\ge r-1\}$. The latter is a proper Zariski closed subset, so $Z\ne Y$ if $Z(s)\ne Y$. Then for every finite set $A\subset Y\setminus Z$ there is an algebraic spray $s_A\colon\cO_Y(-m)^{\oplus r}\to Y$ for some integer $m\ge1$ which is dominating at every point of $A$, hence on a Zariski open neighborhood of $A$. In particular, $Y$ is rationally elliptic if and only if it carries an algebraic spray which is dominating at some point.
\end{theorem}

\begin{proof}
The fiber-dimension locus in the statement is Zariski closed by upper semicontinuity for the projective morphism $\overline{I(s)}\to Y$. Since $\operatorname{codim}I(s)\ge2$ in the smooth variety $E$, $\dim\overline{I(s)}\le\dim Y+r-2$, so the locus where $\dim\overline{I(s)}_y\ge r-1$ has dimension $<\dim Y$.

\emph{Step 1.} For $q\gg0$, $E\otimes\cO_Y(q)$ is generated by global sections \cite[n$^\circ$\,66, Th\'eor\`eme 2(a)]{Se55}, so $r$ general global sections of $E\otimes\cO_Y(q)$ form a basis at every point of $A$. They form a frame on a Zariski open neighborhood of $A$, and $\cO_Y(q)$ is trivial on the complement of a hyperplane avoiding $A$. Hence there is a Zariski open set $U_0\supset A$ over which $E$ is trivial. Taking a generic linear combination of homogeneous polynomials of the same degree vanishing on $Z\cup(Y\setminus U_0)$, we obtain a homogeneous polynomial $h$ of some positive degree $d$ vanishing on $Z\cup(Y\setminus U_0)$ but at no point of $A$. Then $D(h)$ is affine, $A\subset D(h)\subset Y\setminus Z$, $E|_{D(h)}=D(h)\times\C^{r}$, $I(s)|_{D(h)}:=I(s)\cap E|_{D(h)}$ is a closed subset of $D(h)\times\C^{r}$ disjoint from $D(h)\times\{0\}$, and $\dim\overline{I(s)}_y\le r-2$ for $y\in D(h)$.

\emph{Step 2.} For $a\in A$, since $\dim\overline{I(s)}_a\le r-2$, the set
\[
B_a:=\bigl\{[v]\in\PP^{r-1}\colon\C v\cap I(s)_a\ne\emptyset\ \text{or}\ [v:0]\in\overline{I(s)}_a\bigr\}
\]
is a proper Zariski closed subset of $\PP^{r-1}$. Indeed, $B_a$ is the image of $\overline{I(s)}_a$ under the projection from $[0:1]$, which is defined on $\overline{I(s)}_a$ since $0\notin I(s)_a$. Its image has dimension at most $r-2$, proving properness. Hence there is a basis $(v_1,\dots,v_r)$ of $\C^{r}$ with $[v_i]\in\PP^{r-1}\setminus\bigcup_{a\in A}B_a$. For each $i$ let $\pi_i\colon\C^{r}\to\C^{r}/\C v_i\cong\C^{r-1}$ be the projection. Then $[v_i:0]\notin\overline{I(s)}_y$ for $y$ in a Zariski open neighborhood $U_i\subset D(h)$ of $A$. The projection from $[v_i:0]$ is therefore a proper morphism on $\overline{I(s)}|_{U_i}$. The inverse image of the affine chart $U_i\times\C^{r-1}$ of $U_i\times\PP^{r-1}$ under this projection is exactly $I(s)|_{U_i}$: every point at infinity other than the excluded center maps to infinity. Consequently $(\id\times\pi_i)(I(s)|_{U_i})$ is closed in that affine chart. The choice of $v_i$ ensures that it misses $(a,0)$ for every $a\in A$. Since restricting a closure to an open subset is the closure of the restriction, $(a,0)\notin Z_i$, where $Z_i:=\overline{(\id\times\pi_i)(I(s)|_{D(h)})}$ is the closure in $D(h)\times\C^{r-1}$. Since $D(h)$ is affine there is $g_i\in\cO^{\mathrm{alg}}(D(h)\times\C^{r-1})$ vanishing on $Z_i$ with $g_i(a,0)\ne0$ for $a\in A$. The maps
\[
\varphi_i^{(t)}(y,\zeta):=\bigl(y,\ \zeta+t\,g_i(y,\pi_i(\zeta))\,v_i\bigr)\qquad(t\in\C)
\]
are the flows of locally nilpotent vertical derivations. They are automorphisms of $D(h)\times\C^{r}$ over $D(h)$ fixing $I(s)|_{D(h)}$ pointwise, hence preserving its complement. Put $\varphi(y,t):=\varphi_r^{(t_r)}\circ\dots\circ\varphi_1^{(t_1)}(y,0)$. Then $\varphi\colon D(h)\times\C^{r}\to(D(h)\times\C^{r})\setminus I(s)$ is a morphism with $\varphi(y,0)=(y,0)$.

\emph{Step 3.} $s_h:=s\circ\varphi\colon D(h)\times\C^{r}\to Y$ is a morphism (as $s$ is regular off $I(s)$), and $s_h(y,0)=y$.

\emph{Step 4.} \Cref{lem:localization} applied to $s_h$ gives the algebraic spray $s_A$ on $\cO_Y(-m)^{\oplus r}$ with $m=dk$. Since $g_i(a,0)\ne0$ for $a\in A$ and $v_1,\dots,v_r$ are linearly independent, $s_A$ is dominating at every point of $A$. Conversely, every algebraic spray which is dominating at some point has exceptional set different from $Y$.
\end{proof}

\begin{remark}\label[remark]{rem:alg-spray}
The resolution of the indeterminacy of $s$ is not needed: instead of blowing up $I(s)$, one moves inside the complement of the codimension-$2$ set $I(s)$ by shears. A statement similar to Step 2 of the proof of \cref{thm:alg-spray} appears in the concurrent work \cite[Lemma 2.2]{DKZ26}. \Cref{lem:localization} is a version of Gromov's Localization Lemma \cite[3.5.B]{Gr89}, see also \cite[Proposition 2.1]{KZ24b}, whose proof is given in \cite[Appendix C]{KZ24b}, and the extension lemma in \cite{KZ24a}. The spray of \cref{thm:alg-spray} dominates only on a neighborhood of $A$, because of the twist by $h$ and the shears. This is why for approximation we use \cref{prop:normalized}, which keeps the exceptional set $Z(s)$ at the price of rationality away from the zero section. If the set $Z$ of \cref{thm:alg-spray} is empty, then every point of $Y$ has a Zariski open neighborhood on which some algebraic spray on $Y$ is dominating. Hence $Y$ is algebraically subelliptic, and therefore algebraically elliptic \cite[Theorem 1.1]{KZ24a} (see also \cite[Theorem 1.9]{Za25}). It remains open whether $Z(s)=\emptyset$ alone implies algebraic ellipticity.
\end{remark}

\subsection{Birational maps}
Throughout this subsection, $Y$ and $\widetilde Y$ are smooth varieties.

\begin{definition}\label[definition]{def:Zpsi}
Let $f\colon Y\merto\widetilde Y$ be birational. Let $U_f\subset Y$ be the union of all Zariski open sets on which $f$ is regular and an open embedding. Put $\widetilde U_f:=f(U_f)$ and $\widetilde Z_1:=\widetilde Y\setminus\widetilde U_f$. For a Zariski closed set $Z\subsetneq Y$ put
\[
Z_f:=\widetilde Z_1\cup\overline{f(Z\cap U_f)}.
\]
\end{definition}

If $f$ is an open embedding on $U_1$ and on $U_2$ and $f(y_1)=f(y_2)$ with $y_i\in U_i$, then the two inverses are regular representatives of $f^{-1}$ and agree near $f(y_1)$, so $y_1=y_2$. Hence $f|_{U_f}\colon U_f\to\widetilde U_f$ is an isomorphism. Moreover $Z_f$ is a proper Zariski closed subset of $\widetilde Y$, and
\begin{equation}\label{eq:Zpsi}
f^{-1}(\widetilde Y\setminus Z_f)\subset U_f\setminus Z,\qquad \emptyset_f=\widetilde Z_1.
\end{equation}

\begin{theorem}\label[theorem]{thm:transport}
Let $f\colon Y\merto\widetilde Y$ be birational and $s\colon Y\times\C^{r}\merto Y$ a rational spray on a trivial bundle with $Z(s)\ne Y$. Then $\tilde s:=f\circ s\circ(f^{-1}\times\id)\colon\widetilde Y\times\C^{r}\merto\widetilde Y$ is a rational spray with $Z(\tilde s)\subset Z(s)_f$. For a rational spray $s\colon E\merto Y$ with $Z(s)\ne Y$ on an arbitrary bundle and a nonempty Zariski open set $U$ over which $E$ is trivial, the same holds for $s_U$ \textup{(}\cref{rem:rational-basic}\textup{(2)}\textup{)}, with $Z(s_U)_f\subset(Z(s)\cup(Y\setminus U))_f$.
\end{theorem}

\begin{proof}
Since $f^{-1}\times\id$ is birational and $s$ is dominant \textup{(}as $s\circ0_E=\id_Y$\textup{)}, the composition $\tilde s$ is a well-defined rational map. Let $\tilde y\in\widetilde Y\setminus Z(s)_f$, which is a nonempty Zariski open set. By \eqref{eq:Zpsi}, $y:=f^{-1}(\tilde y)\in U_f\setminus Z(s)$, so $s$ is regular at $(y,0)$ with $s(y,0)=y\in U_f$, and $\zeta\mapsto s(y,\zeta)$ is a submersion at $0$. Since $f^{-1}$ is regular near $\tilde y$ and $f$ is regular and a local isomorphism near $y$, $\tilde s$ is regular at $(\tilde y,0)$, $\tilde s(\tilde y,0)=f(y)=\tilde y$, and $\zeta\mapsto\tilde s(\tilde y,\zeta)=f(s(y,\zeta))$ is a submersion at $0$. Hence $\tilde s$ is a rational spray and $\tilde y\notin Z(\tilde s)$. The last statement follows from $Z(s_U)\subset Z(s)\cup(Y\setminus U)$, which is different from $Y$ since $Y$ is irreducible.
\end{proof}

\begin{proof}[Proof of \cref{thm:ME}]
Trivialize a rational spray over a nonempty Zariski open set (\cref{rem:rational-basic}(2)) and apply \cref{thm:transport} to obtain birational invariance. Projective space carries the spray of \cref{ex:homogeneous}(a), so every rational variety is rationally elliptic. Conversely, a dominating fiber of a rational spray gives a dominant rational map $\C^{r}\merto Y$, proving unirationality. The projective equivalence is \cref{thm:alg-spray}.
\end{proof}

It remains open whether every smooth unirational projective variety is rationally elliptic. This is a weak form of Gromov's question whether such varieties are algebraically elliptic, see \cite[Question 4.1]{Za25}.

\section*{Acknowledgments}
The first author was supported by JSPS KAKENHI Grant Number JP24K16919.

\section*{AI disclosure}
The authors proved the nonparametric case of \cref{thm:main1} without the use of AI. Its parametric version, and the birational invariance results of \cref{sec:birational} based on rational sprays, were obtained with the assistance of Claude Opus 5.5 (Anthropic). Generative AI tools were also used in the preparation and revision of this manuscript.
The authors take full responsibility for its content.

\end{document}